\documentclass{article}
\usepackage[a4paper, margin=2 cm]{geometry}
\usepackage[utf8]{inputenc}
\usepackage{mathtools}
\usepackage{amsmath}
\usepackage{amssymb}
\usepackage{amsthm}
\usepackage{bbm}
\usepackage{enumerate}  
\usepackage{tikz}
\usepackage{hyperref, cleveref}
\hypersetup{
    colorlinks=true,
    linkcolor=blue,
    filecolor=magenta,      
    urlcolor=cyan,
    breaklinks=true,
    }

\usepackage[dvipsnames]{xcolor}
\usepackage{lineno}
\usepackage{fullpage}
\usepackage{enumitem,thmtools}
\usepackage{comment}

\usepackage{latexsym}
  
\def\PP{{\mathbb P}}
\def\NN{{\mathbb N}}
  
\def\EE{{\mathbb E}}

\def\im{{\text{Im}}}

\DeclarePairedDelimiter{\ceil}{\lceil}{\rceil}

\DeclareMathOperator{\se}{\subseteq}
\renewcommand{\epsilon}{\varepsilon}
\DeclareMathOperator{\eps}{\varepsilon}
\newcommand{\red}[1]{#1_{\text{red}}}
\newcommand{\blue}[1]{#1_{\text{blue}}}

\newtheorem{theorem}{Theorem}[section]
\newtheorem{lemma}[theorem]{Lemma}
\newtheorem{question}[theorem]{Question}
\newtheorem{corollary}[theorem]{Corollary}
\newtheorem{proposition}[theorem]{Proposition}
\newtheorem{claim}[theorem]{Claim}
\newtheorem{fact}[theorem]{Fact}

\newtheorem{definition}[theorem]{Definition}

\title{A degree version of the Burr--Erd\H{o}s conjecture on trees}
\author{Jasmin Katz\thanks{
		Department of Mathematics,
		London School of Economics and Political Science, Houghton Street,
		London WC2A 2AE, UK.
		Email: \href{mailto:j.katz3@lse.ac.uk}{\texttt{j.katz3}@\texttt{lse.ac.uk}.
	}} \and Mat\'ias Pavez-Sign\'e\thanks{Departamento de Ingenier\'ia Matem\'atica, Universidad de Chile, and Centro de Modelamiento Matem\'atico (CNRS IRL2807), Chile. Supported by ANID Fondecyt Regular grant No. 1241398 and ANID Basal Grant CMM FB210005.  \emph{E-mail:} \href{mailto:mpavez@dim.uchile.cl}{\texttt{mpavez}@\texttt{dim.uchile.cl}.}} \and Jozef Skokan\thanks{
		Department of Mathematics,
		London School of Economics and Political Science, Houghton Street,
		London WC2A 2AE, UK.
		Email: \href{mailto:j.skokan@lse.ac.uk}{\texttt{j.skokan}@\texttt{lse.ac.uk}.
	}}}
\date{}
  \makeatletter
  \newcommand{\labelinthm}[1]{%
     \label{temp#1}
     \protected@write \@auxout {}{\string \newlabel{#1}{{\emph{\ref{temp#1}}}{\thepage}{\emph{\ref{temp#1}}}{temp#1}{}} }%
  }
  \makeatother
\newcounter{propcounter}

\begin{document}

\maketitle
\begin{abstract}
An old conjecture of Burr and Erd\H os states that the Ramsey number of any $n$-vertex tree $T$ is at most $2n-2$. 
In 2012, Schelp asked whether a degree version of the Burr--Erd\H{o}s conjecture holds. More precisely, Schelp asked if is it true that for any $\varepsilon>0$ and $\Delta\ge 2$, if $G$ is a graph on $N\ge (2+\varepsilon)n$ vertices and minimum degree $\delta(G)\ge \lfloor 3N/4\rfloor$, then every blue/red colouring of the edges of $G$ yields a monochromatic copy of each $n$-vertex tree with maximum degree at most $\Delta$. We prove this conjecture in a strong form, showing that it is true even if one removes the extra $\varepsilon n$ term in the size of the host graph.
\end{abstract}
\section{Introduction}

Ramsey theory is one of the most fundamental areas of research in extremal combinatorics. We say a graph $G$ is {Ramsey} for a graph $H$ (or {$H$-Ramsey}) if every blue/red colouring of the edges of $G$ yields a monochromatic copy of $H$, in which case we write $G \rightarrow H$ to denote that $G$ is $H$-Ramsey.  From an extremal graph theory perspective, the aim is to understand how large the \textit{minimal} $H$-Ramsey graphs are, where minimality typically refers to either the number of vertices or the number of edges among all $H$-Ramsey graphs. The classical problem in the area involves determining the {Ramsey number} of a graph $H$, denoted $R(H)$, which is the smallest integer $n\in\mathbb N$ such that $K_n$ is Ramsey for $H$. In other words, $R(H)$ is the least order among all $H$-Ramsey graphs, which is well-defined as a consequence of Ramsey's theorem~\cite{ramsey30} from 1929. Determining the Ramsey number of a graph $H$ is an extremely difficult problem in general, and even proving good bounds on $R(H)$ is challenging. For instance, it took nearly 70 years to give the first exponential improvement on $R(K_t)$ (see e.g.~\cite{balister2024upper,campos2023exponential,gupta2024optimizing}).

The progress on computing exact Ramsey numbers of graphs is also very modest.  We know $R(H)$ only for a few classes of graphs, such as factors~\cite{bucic2023tight}, paths~\cite{gerencser1967ramsey}, and cycles~\cite{BONDY197346,faudree1974all,rosta1973ramsey}. For general trees,  Burr and Erd\H os~\cite{BurrErdosRamsey} conjectured in 1976 that any $n$-vertex tree $T$ satisfies $R(T)\le 2n-2$ for even $n$ and $R(T)\le 2n-3$ for $n$ odd, which would be tight as witnessed by the $n$-vertex star~\cite{Harary1972}. Though this conjecture is still open, we know by a result of Yi Zhao \cite{zhao2011} that $R(T)\le 2n-2$ holds for every $n$-vertex tree $T$, whenever $n$ is sufficiently large. 

    In this paper, we are interested in understanding how resilient the property of being $H$-Ramsey is under adversarial local edge-deletion, a question that was first considered in 2010 by Li, Nikiforov, and Schelp~\cite{li2010new}. Formally, the question we want to address is the following.
    \begin{question}\label{mainquestion}
    For a graph $H$, does there exist a constant $c_H\in (0,1)$ such that if $G$ is a graph with at least $R(H)$ vertices and minimum degree $\delta(G)\ge c_H|G|$, then $G\to H$?
    \end{question}
    In 2012, Schelp~\cite{schelp2012} asked for which graphs $H$ one can find a non-trivial answer to Question~\ref{mainquestion}, and proposed a series of conjectures about which families of graphs are good candidates for solving it. This problem has received considerable attention since then. For example,  Benevides, Luczak, Scott, Skokan, and White~\cite{cycles_2012} showed that $c=3/4$ is the answer for the $n$-vertex cycle $C_n$ (see also~\cite{balogh2022monochromatic,li2010new,paths_schelp_conj, NIKIFOROV200869}), and, very recently, Balogh, Freschi, and Treglown~\cite{balogh2025ramsey} made progress in the case when $H$ consists of vertex-disjoint copies of $K_3$. In the asymmetric setting, Arag{\~a}o, Marciano, and Mendon\c{c}a~\cite {aragao2025degree} studied an analogue of Question~\ref{mainquestion} for the pair $(P_n,K_t)$, where $P_n$ denotes the path of length $n$. In the case of trees, Schelp~\cite{schelp2012} conjectured that $c=3/4$ would work for bounded-degree trees too, asking the following degree version of the Burr--Erd\H os conjecture on trees.
   
   \begin{question}Let $\Delta\ge 2$, let $\varepsilon\in (0,1)$, and let $G$ be a graph with $(2 + \epsilon)n$ vertices and minimum degree $\delta(G) > 3|G|/4$. If $T$ is an $n$-vertex tree with maximum degree $\Delta(T)\le \Delta$, does $G \rightarrow T$?\end{question}
   In this paper, we answer this question in the affirmative and show that, in fact, the $\epsilon n$ term is not necessary. 
\begin{theorem}\label{thm:main}
    For all $\Delta \in \NN$ there is some $n_0 \in \NN$ such that for all $n \ge n_0$ the following holds. 
    
    If $G$ is a blue/red edge-coloured graph with $|G|= 2n-1$ and $\delta(G) \ge \lfloor \frac{3|G|}{4} \rfloor$, then $G$ contains a monochromatic copy of every $n$-vertex tree with maximum degree at most $\Delta$.
\end{theorem} 
We also show that with an additional constant term in the minimum degree, we can embed a bounded degree tree in a graph on $2n-2$ vertices, which is tight up to the constant term.
\begin{theorem}\label{thm:main3}
    For all $\Delta \in \NN$ there is some $C>0$ and $n_0 \in \NN$ such that for all $n \ge n_0$ the following holds. 
    
    If $G$ is a blue/red edge-coloured graph with $|G|= 2n-2$ and $\delta(G) \ge \lfloor \frac{3|G|}{4} \rfloor + C$, then $G$ contains a monochromatic copy of every $n$-vertex tree with maximum degree at most $\Delta$.
\end{theorem} 
The extremal example of Theorem~\ref{thm:main3} is given by a graph $G$ on $2n-2$ vertices with vertex set $V(G)=V_1\cup V_2\cup  V_3\cup V_4$, where $|V_1|=|V_3|=n/2$ and $|V_2|=|V_4|=n/2-1$ (assuming $n$ is even). For each $i\in [4]$, $V_i$ is any blue/red coloured clique. Both $G[V_1, V_2]$ and $G[V_3, V_4]$ are complete bipartite graphs, with all edges coloured blue, and both $G[V_1, V_4]$ and $G[V_2, V_3]$ are complete bipartite graphs with all edges coloured red. The minimum degree of $G$ is $3n/2 - 3  = \lfloor \frac{3|G|}{4}\rfloor -1$ (this is the degree of the vertices in $V_1$ and $V_3$), but each monochromatic component covers only $n-1$ vertices, so it contains no tree with $n$ vertices.

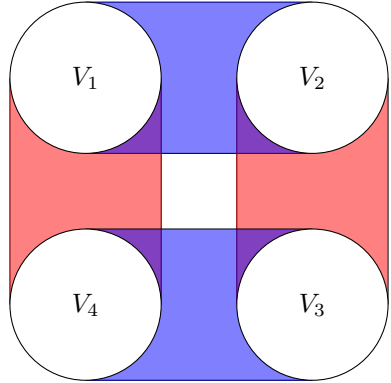
\begin{figure}[h!]
\centering{
\begin{tikzpicture}

    % Draw lines connecting the top and bottom of the circles (top part)
    \draw (0,1) -- (3,1); % top line
    \draw (0,-1) -- (3,-1); % bottom line
    \draw (-1,0) -- (-1,-3); % left line
    \draw (1,0) -- (1,-3); % left right line
    \draw (2,0) -- (2,-3); % left right line
    \draw (4,0) -- (4,-3); % right right line
    
    % Fill the area between the two circles and lines with blue (top part)
    \fill[red, opacity=0.5] (-1,0) -- (1,0) -- (1,-3) -- (-1,-3) -- cycle;
    \fill[red, opacity=0.5] (2,0) -- (4,0) -- (4,-3) -- (2,-3) -- cycle;
    \fill[blue, opacity=0.5] (0,1) -- (0,-1) -- (3,-1) -- (3,1) -- cycle;

    % Draw lines connecting the top and bottom of the circles (bottom part)
    \draw (0,-2) -- (3,-2); % bottom line of the lower section
    \draw (0,-4) -- (3,-4); % top line of the lower section
    
    % Fill the area between the two circles and lines with blue (bottom part)
    \fill[blue, opacity=0.5] (0,-2) -- (0,-4) -- (3,-4) -- (3,-2) -- cycle;

    \coordinate (A) at (0,0);
    \coordinate (B) at (3,0);
    \coordinate (C) at (0,-3);
    \coordinate (D) at (3,-3);

    % Draw two circles (top part)
    \draw[fill=white] (A) circle(1);
    \draw[fill=white] (B) circle(1);
    
    % Draw two circles (bottom part)
    \draw[fill=white] (C) circle(1);
    \draw[fill=white] (D) circle(1);

    \node at (A) { \(V_1\)};
    \node at (B) {\(V_2\)};
    \node at (D) {\(V_3\)};
    \node at (C) {\(V_4\)};
\end{tikzpicture}
}\caption{In this picture is depicted the extremal colouring, where, for $i\in [4]$, $V_i$ is an arbitrarily coloured clique.}
\label{fig:1}
\end{figure}
We made no attempt to optimise the constant $C$ in Theorem~\ref{thm:main3}, though our proof gives $C=\text{poly}(\Delta)$. Moreover, if we allow a linear error term in the minimum degree, we can improve the maximum degree of the trees we can embed, as follows.
\begin{theorem}\label{thm:main2}
    For all $\delta>0$, there are $\alpha\in(0,1)$ and $n_0 \in \NN$ such that for all $n \ge n_0$ the following holds. 
    
    If $G$ is a blue/red edge-coloured graph with $|G|= 2n-2$ and $\delta(G) \ge (3/2+\delta)n$, then $G$ contains a monochromatic copy of every $n$-vertex tree with maximum degree at most $n^{\alpha}$.
\end{theorem} 
The bound on the maximum degree in Theorem~\ref{thm:main2} might not be optimal, and we believe it might be pushed up to $o(n/\log n)$, which would be best possible. To see why this is a natural barrier, let us consider the following example, which is inspired by a famous construction of Koml\'os, S\'ark\"ozy and Szemer\'edi~\cite{KSS}. Let $X$ and $Y$ be disjoint sets of size $n$ each, and let $G$ be a graph with vertex set $V(G)=X\cup Y$ defined as follows. For $p=0.9$, say, let $G[X]$ and $G[Y]$ be independent copies of $G(n,p)$, and between $X$ and $Y$ put a binomial bipartite random graph with edge probability $p$. It is clear that with high probability $G$ has minimum degree roughly $0.9|G|$. Now, colour the edges in both $G[X]$ and $G[Y]$ with blue and colour the crossing edges between $X$ and $Y$ in red. For a sufficiently small constant $c\in (0,1)$, if $T$ is the tree consisting of a vertex that has $c\log n$ children and each grandchild has roughly $n/c\log n$ children, then with high probability $T$ is not a monochromatic subgraph of $G$, as otherwise there would be a dominating set of size $c\log n$ in some colour, which does not exist with high probability. 

    For the proofs of Theorems~\ref{thm:main},~\ref{thm:main3} and~\ref{thm:main2}, we follow a stability approach as we briefly outline now. Given a graph $G$ as in Theorem~\ref{thm:main}, the first step is to prove a stability result (Theorem~\ref{thm:stability}) which states that either $G$ contains each $n$-vertex tree with bounded degree in some colour or it is close to the extremal example depicted in Figure~\ref{fig:1}. To prove this, we use the \textit{regularity method} and, in particular, ideas around the \textit{connected-matching method} introduced by \L uczak~\cite{luczak1999r} (see Section~\ref{sec:regularity}) in combination with structural results for graphs without large connected matchings (see Lemma~\ref{lem:H_strcuture}). The second step is thus to prove that if $G$ looks similar to the extremal example, then it contains every $n$-vertex tree with bounded degree in some colour. For the extremal analysis, we assume that $G$ contains pairwise disjoint sets $V_1,V_2,V_3,V_4$ such that $G[V_i]$ is an almost complete graph of size $|V_i|\approx n/2$, for $i\in [4]$, $G[V_1,V_2]$ and $G[V_3,V_4]$ are blue almost complete bipartite graphs, $G[V_1,V_4]$ and $G[V_2,V_3]$ are red almost complete bipartite graphs, and there are only $o(n^2)$ crossing edges. For simplicity, assume that $V(G)=V_1\cup V_2\cup V_3\cup V_4$, in which case, since $|G|=2n-1$, at least one of the monochromatic graphs covers at least $n$ vertices. The problem thus reduces to finding an $n$-vertex tree in a graph of high minimum degree, for which we use randomised embeddings and the \textit{absorption method} via switchings. 
    
    Of course, the above outline is oversimplified. In particular, a lot of extra work is needed to refine the approximate extremal example we start with, as $V_1\cup V_2\cup V_3\cup V_4$ might not cover $V(G)$ and it is crucial to incorporate every single vertex in this structure. For Theorem~\ref{thm:main3}, we need some extra work as the components we find might have exactly $n-1$ vertices, in which case we need to use the extra room in the minimum degree condition to find crossing edges between two of the monochromatic components. If the number of crossing edges is large enough (compared to the maximum degree of the tree), we show that this is enough to find a copy of each $n$-vertex tree with bounded degree. For trees with unbounded degree as in Theorem~\ref{thm:main2}, we follow a similar strategy, and we highlight that for the extremal analysis we need some sort of bipartite version of a celebrated result of Koml\'os, S\'ark\"ozy and Szemer\'edi~\cite{KSS} about minimum degree conditions forcing the containment of spanning trees.

    The paper is organised as follows. In Section~\ref{sec:preliminaries} we collect the necessary tools for reading this paper. In Section~\ref{sec:regularity} we show how to embed bounded-degree trees using regularity and we use these results in Section~\ref{sec:stability}  to prove the stability results. In Section~\ref{sec:extremal} we prove four results needed to analyse the extremal situation and in Section~\ref{sec:final} we finally prove Theorems~\ref{thm:main},~\ref{thm:main3} and~\ref{thm:main2}. We finish with some concluding remarks and questions in Section~\ref{sec:remarks}.
\section{Preliminaries}\label{sec:preliminaries}
\subsection{Notation}
For a graph $G$,  $V(G)$ and $E(G)$ denote the set of vertices and edges $G$, respectively, and we write $|G|$ for the number of vertices of $G$ and $e(G)$ for the number of edges of $G$. Given a vertex $x \in V(G)$, we write $N_G(x)$ for its neighbourhood in $G$ and  $d_G(x):=|N_G(x)|$ is the degree of $x$ in $G$. We let $\delta(G)$ and $\Delta(G)$ be the minimum and maximum degree of $G$, respectively, and write $d(G)=2e(G)/|G|$ for the average degree of the vertices in $G$. Let $X$ and $Y$ be two disjoint subsets of $V(G)$. We write $N_G(X) = \bigcup_{x \in X}N_G(x)$ to be the union of all neighbourhoods of vertices in $X$. We denote by $G[X]$ the subgraph of $G$ induced by the vertices of $X$, and $E(X, Y)$ the set of edges joining $X$ and $Y$. We write $G[X, Y]$ for the bipartite subgraph of $G$ with bipartition $X \cup Y$ and edge set $E(X, Y)$. If $G$ is bipartite, we write $G = (A, B)$ to indicate $A$ and $B$ are the bipartition classes of $G$, with $|A| \ge |B|$. Say a set $X\subset V(G)$ is $k$-independent if every pair of distinct vertices in $X$ are at distance at least $k$ in $G$.

Suppose $G$ has a blue/red colouring. We let $\red{G}$ and $\blue{G}$ denote the graphs spanned by the red edges and the blue edges, respectively. We write $\red{d}(x)$ for the degree of $x$ in $\red{G}$ and $\blue{d}(x)$ for the degree of $x$ in $\blue{G}$, and use similar notations for the red and blue neighbourhoods of a vertex or a set of vertices.

As usual, we write $[n]=\{1,\ldots,n\}$ for $n\in\mathbb N$ and, for $a,b,c\in\mathbb R$, we use $a=b\pm c$ to denote that $a\in [b-c,b+c]$. We will use the standard hierarchy notation, that is, for $a,b\in (0,1]$, we will use $a\ll b$ to mean that there exists a non-decreasing function $f:(0,1]\to (0,1]$ such that if $a\le f(b)$ then the subsequent statement holds. For $a,b\geq1$, we write $a\ll b$ if $1/b\ll1/a$. Hierarchies with more constants are defined in a similar way, and, for simplicity, we will sometimes ignore floor and ceiling signs when doing so does not affect the argument.

\subsection{Probabilistic tools}
\begin{lemma}[Chernoff's bound~{\cite[Corollary 2.3, Theorem 2.10]{JLR2000}}]\label{lemma:chernoff}
Let $X$ be either a binomial random variable or a hypergeometric random variable. Then, for all $0<\eps\le 3/2$,
\[\mathbb{P}\left(\big|X-\mathbb E[X]\big|\ge \eps\mathbb E[X]\right)\le 2\exp(-\eps^2\mathbb{E}[X]/3).\]
\end{lemma}

\begin{lemma}[McDiarmid's inequality~{\cite[Lemma 1.2]{McDiarmid_1989}}]\label{lemma:mcdiarmid}
Let $X_1,\ldots,X_m$ be independent random variables taking values in a set $\Omega$. Let $c_1,\ldots,c_m\geq 0$ and suppose $f:\Omega^m\to\mathbb{R}$ is a function such that for every $i\in[m]$ and every $x_1,\ldots,x_m,x_i'\in\Omega$, we have $\Big|f(x_1,\ldots,x_i,\ldots,x_m)-f(x_1,\ldots,x_i',\ldots,x_m)\Big|\leq c_i$.
Then, for all $t>0$,
\[\mathbb{P}\left(\Big|f(X_1,\ldots,X_m)-\mathbb{E}[f(X_1,\ldots,X_m)]\Big|\geq t\right)\leq2\exp\left(\frac{-2t^2}{\sum_{i=1}^mc_i^2}\right).\]
\end{lemma}
A sequence of random variables  $(X_i)_{i \geq 0}$ is a submartingale if $\mathbb E[X_{i+1} \mid X_i, \ldots, X_0] \geq X_i$ for each $i\geq 0$. We will use the following Azuma-type bound for submartingales.

\begin{lemma}[Azuma's inequality~\cite{wormald1999differential}]\label{lem:azuma}
Let $(X_i)_{i \geq 0}$ be a submartingale and let $c_i>0$ for each $i\geq 1$. If $|X_i -X_{i-1}| <c_i$ for each $i\geq 1$, then, for each $n\geq 1$,
\[
\mathbb{P}(X_n-X_0\leq t) \leq 2 \exp \left(-\frac{t^2}{2\sum_{i=1}^nc_i^2} \right).
\]
\end{lemma}
\subsection{Trees}

A rooted tree $T$ is a tree in which a vertex $r\in V(T)$ has been identified as the root. This induces a partial order on the vertices of $T$, where for each $xy \in E(T)$ we say $x$ is the \textit{parent} of $y$ ($y$ is the \textit{child} of $x$) if $x$ lies on the unique path from $y$ to $r$.
\begin{fact}\label{fact:BddDegTreeBipartition}Let $T$ be an $n$-vertex tree with $\Delta(T)\le \Delta$, and suppose $T$ has bipartition classes of size $n_1$ and $n_2$, with $n_1 \ge n_2$. Then, $n_1 \le \frac{\Delta -1}{\Delta}n + \frac 1 \Delta $ and $n_2 \ge \frac{n-1}{\Delta}$.
\end{fact}

The next few results allow us to find a cut vertex/edge that separates the tree into parts with controlled sizes.
%A \textit{bare path} in a tree $T$ is a path whose internal vertices all have degree 2 in $T$. We need the following result, which says that a tree either has many leaves or many bare paths.
%\begin{lemma}[Lemma~2.1 in~\cite{krivelevich2010embedding}]\label{lem:barepaths}Let $n,k,\ell\in\mathbb N$ and let $T$ be an $n$-vertex tree. If $T$ has at most $\ell$ leaves, then $T$ contains at least $n/(k+1)-2\ell$ vertex-disjoint bare paths, each of length $k$. 
%\end{lemma}

\begin{lemma}[Lemma 4.2 in~\cite{Besomi2019}]\label{lem:cutVertex}	Let $T$  be an $n$-vertex tree, and let $x$ be a leaf of $T$. Then, $T$ has a vertex $z$ such that every component of $T-z$ has at most $\lfloor\frac {n-1}{2}\rfloor$ vertices, except for the component containing $x$ which has at most $\lceil\frac {n-1}{2}\rceil$ vertices.
\end{lemma}

\begin{lemma}\label{lem:cutvertex:2}Let $1/n\ll\gamma\le 1$ and let $T$ be an $n$-vertex tree. Then, there is a vertex $z$ and edge-disjoint subtrees $T_1,T_2\subset T$ such that $\gamma n\le |T_1|\le 2\gamma n$, $T_1$ and $T_2$ intersects exactly at $z$, and $T=T_1\cup T_2$.\end{lemma}
\begin{proof} Let $t$ be an arbitrary vertex we choose as the root of $T$. Let $z$ be a vertex a maximal distance from $t$ such that the tree $T_0$ formed by $z$ and all vertices below $z$ in the ordering given by $t$ satisfies $|T_0|\ge \gamma n$. Let $S_1,\ldots,S_m $ be the components of $T_0-z$ and note that, by maximality of $z$, $|S_i|<\gamma n$ for each $i\in [m]$. Let $I\subset [m]$ be minimal subject to $|\cup_{i\in [m]}S_i|+1\ge \gamma n$, nothing that $I$ is well-defined as $|T_0|\ge \gamma n$. The tree $T_1$  formed by $z$ and $\cup_{i\in I}S_i$ satisfies the conclusions, as $|T_1|> 2\gamma n$ would contradict the minimality of $I$.\end{proof}

The following three lemmas apply only to trees of bounded maximum degree. The first allows us to find a single edge connecting a subtree of desired size to the rest of the tree. 

\begin{lemma}\label{lem:cutedge:0}Let $\Delta\ge 2$ and $1/n\ll\gamma\ll1/\Delta$, and let $T$ be an $n$-vertex tree with $\Delta(T)\le \Delta$. Then, there is an edge $e\in E(T)$ and a component $T'$ of $T-e$ such that $\gamma n\le |T'|\le \Delta\gamma n$.\end{lemma}

\begin{proof} Let $t$ be an arbitrary vertex we choose as the root of $T$. Let $z$ be a vertex a maximal distance from $t$ such that the tree $T_0$ formed by $z$ and all vertices below $z$ in the ordering given by $t$ satisfies $|T_0|> \Delta\gamma n$. Let $S_1,\ldots,S_m $ be the components of $T_0-z$ and note that, by maximality of $z$, $|S_i|\le \Delta \gamma n$ for each $i\in [m]$. Also note that as $z$ has degree at most $\Delta$, one of these components $S_j$ has $|S_j| \ge \gamma n$. Choosing $s$ to be the edge between $z$ and the root of $S_j$ gives the desired property, with $T' = S_j$.
\end{proof}

\begin{lemma}[Lemma 4.4 in~\cite{Besomi2019}]\label{lem:numberPartition}
Let $m,n \in \NN$ and let $(a_i)_{i\in [m]}$ be a sequence of integers  such that $0 < a_i \leq \lceil n/2\rceil$, for each $i\in [m]$, and $\sum_{i\in [m]} a_i \le n$. Then, there exists a partition $[m]=I_1\cup I_2\cup I_3$ such that $I_3$ contains at most one element and 
\[\sum_{i\in I_3} a_i \leq \sum_{i\in I_2} a_i \leq \sum_{i\in I_1} a_i\leq \lceil n/2\rceil.\]
\end{lemma}

\begin{corollary}\label{cor:t/2_forest}
Let $n,\Delta\in \NN$ and let $T$ be an $n$-vertex tree with $\Delta(T) \le \Delta$. Then, there exists a vertex $z\in V(T)$ and vertex-disjoint forests $F_1,F_2\subset T$ such that
\begin{enumerate}[label=\upshape{(\roman{enumi})}]
    \item $T-z = F_1 \cup F_2$,
    \item $|F_1|\le \lceil n/2 \rceil$, and
    \item $F_2$ has a proper two-colouring such that the largest colour class has order at most $\frac{(\Delta-1)n}{2\Delta}+ \frac 1 {\Delta}$.
\end{enumerate}  
\end{corollary}

\begin{proof}
        Let $T$ be an $n$-vertex tree $\Delta(T)\le\Delta$. Then, Lemma \ref{lem:cutVertex} gives a vertex $z \in V(T)$ such that every component of $T-z$ has at most $\lceil n/2\rceil$ vertices. Let $C_1, \ldots ,C_m$ be the components of $T-z$ for some $m \le \Delta$. By Lemma \ref{lem:numberPartition}, there are pairwise disjoint subsets $I_1,I_2,I_3\subset [m]$ with $\bigcup_{i\in [3]} I_i = [m]$ such that
    \begin{equation*}\label{eq:3PartitionSize}
        \sum_{i\in I_3} |C_i| \leq \sum_{i\in I_2} |C_i| \leq \sum_{i\in I_1} |C_i|\leq \lceil n/2\rceil.
    \end{equation*}
    Let $H_j = \bigcup_{i \in I_j } C_i$ for $j \in [3]$ and note that $F_1 = H_1$ clearly satisfies the required conditions. Now, suppose we have any 2-colouring of $H_2$ and $H_3$. Letting $\{i, j\} = \{2, 3\}$ and using Fact \ref{fact:BddDegTreeBipartition}, the larger colour class of $H_i$ and the smaller colour class of $H_j$ add up to at most 
    \begin{align*}
    \frac 1 \Delta +\frac{\Delta-1}{\Delta}|H_i| + \frac{|H_j|}{2} & \le \frac{(\Delta-1)(|H_i| + |H_1|)}{2\Delta} + \frac{|H_j|}{2} + \frac 1 \Delta \\
    & = \frac{(\Delta-1)(|H_i| + |H_1|)}{2\Delta} + \frac{n - (|H_i| + |H_1|)}{2} + \frac 1 \Delta \\
    & \le \frac{n}{2} + (|H_i| + |H_1|)\frac{1}{2}\left(\frac{\Delta -1}{\Delta} - 1 \right) +\frac 1 \Delta\\
    & \le \frac{\Delta -1}{\Delta}\frac{n}{2}+ \frac 1 \Delta.
    \end{align*}
    This gives the required 2-colouring of $F_2$,
\end{proof}
The final results of this section partition a tree into a few small pieces.
\begin{definition} For $\beta\in(0,1)$, a $\beta$-decomposition of an $n$-vertex tree $T$ is a collection of pairwise edge-disjoint subtrees $T_1,\ldots, T_s\subset T$ such that 
\begin{enumerate}[label=\upshape{(\roman{enumi})}]
    \item\label{p:betadecomp:1} $E(T)=E(T_1)\cup\ldots\cup E(T_s)$, and
    \item\label{p:betadecomp:2} $|T_i|\le\beta n$ for all $i\in [s]$.
\end{enumerate}
\end{definition}

\begin{lemma}\label{beta_decomp_v2}
    Let $1/n\ll\beta<1$. Then, every $n$-vertex tree has a $\beta$-decomposition into at most $100\beta^{-2}$ subtrees.
\end{lemma}
\begin{proof}Starting with $T^0 = T$, at step $i\ge 1$ and while $|T^{i-1}|>\beta n$, use Lemma~\ref{lem:cutvertex:2} to find a vertex $z_i\in V(T^{i-1})$ and  edge-disjoint subtrees $T_i,T'_{i}\subset T^{i-1}$ such that $\beta|T^{i-1}|/2\le |T_i|\le \beta |T^{i-1}|$, $T_i$ and $T_i'$ intersect at $z_i$, and $T_{i-1}=T_i\cup T_i'$. Define $T^i$ by contracting $T_i$ into $z_i$. If $|T^i|\le \beta n$, then set $T_i=T^i$ and stop. This process stops with a collection of pairwise edge-disjoint subtrees $T_1,\ldots, T_s\subset T$ covering $T$  such that for $|T_i|\le \beta |T^{i-1}|\le\beta n$, for $i\in[s-1]$, and $|T_s|\le\beta n$  by definition. Moreover, for $i\in [s-1]$ we have $|T_i|\ge \beta |T_{i-1}|/2\ge \beta^2n/2$, which implies $s\le 100\beta^{-2}$.
\end{proof}
For trees with bounded degree, we can further require that the subtrees are pairwise vertex-disjoint and the connecting vertices are far apart in the tree.  If $T$ is a tree with root $r$, for $v\in V(T)$ we define $T(v)$ as the tree induced by all vertices $u\in V(T)$ that are below $v$ with respect to $r$.

\begin{lemma}\label{vtx_disjoint_beta_decomp}
    Given $\beta \in (0,1)$ and $\Delta ,k\in \NN$, let $n > {4k\Delta^{k+1}}/{\beta}$. Let $T$ be an $n$-vertex with $\Delta(T)\le \Delta$. Then there are pairwise vertex-disjoint rooted subtrees $T_1,\ldots, T_s$ such that the following holds.
    \begin{enumerate}[label=\upshape{(\roman{enumi})}]
        \item $V(T)=V(T_1)\cup\ldots\cup V(T_s)$.\label{decomp1}
        \item Letting $r_i$ be the root of the subtree $T_i$ for each $i\in [s]$, the set $\{r_1, \ldots, r_s\}$ is $k$-independent. \label{decomp2}
        \item $\beta n/ (2\Delta)\le |T_i| \leq \beta n $ for each $i\in [s]$.\label{decomp3}
        \item $s\le 2\Delta\beta^{-1}$. \label{decomp5}
    \end{enumerate}
\end{lemma}
\begin{proof}
    We iteratively construct the subtrees $T_1,\ldots, T_s$ as follows. Letting $T^0 = T$ and root $T^0$ at an arbitrary vertex $r$. At step $i\ge 1$, let $r_i$ be a vertex with maximum distance from the $r$ such that $|T^{i-1}(r_i)| \ge \beta n/(2\Delta)$. By maximality, for each child $s$ of $r_i$, we have $|T(s)| < \beta n /(2\Delta)$, and thus $|T^{i-1}(r_i)| \le \beta n/2$. Let $T^i = T^{i-1}- T^{i-1}(r_i)$ and $T_i=T^{i-1}(r_i)$. Note that if a vertex $x$ has distance at most $k$ from some $r_1,\ldots, r_{i-1}$, then $|T^{i-1}(x)| \le 1+\Delta+\ldots+\Delta^k\le (k+1)\Delta^k < \beta n/2\Delta$, which implies $r_i$ is at distance at least $k$ from each $r_1,\ldots, r_{i-1}$ and so $\{r_1,\ldots, r_i\}$ is $k$-independent. We stop the process at step $j$ when $|T^j| \le \beta n/2\Delta$. It may be that $r_j$ is not $k$-independent from $r_1,\ldots, r_{j-1}$, in which case we replace $T_{j-1}$ with $T_{j-1} \cup T^{j}$ and root it at $r_{j-1}$. Properties \ref{decomp1}--\ref{decomp5} hold by construction.
\end{proof}

\subsection{Tree embeddings}
The next two results are greedy embedding lemmas that we will use without any further reference. 
\begin{lemma}Let $G$ be a graph with $\delta(G)\ge n-1$ and let $T$ be an $n$-vertex tree. Then, for every $t\in V(T)$ and $v\in V(G)$, there is a copy of $T$ in $G$ where $t$ is copied to $v$.\end{lemma}
\begin{lemma}Let $T$ be a tree with bipartition $V_1, V_2$ such that $|V_i|=t_i$, $i\in [2]$. Suppose $G$ is a bipartite graph with parts $U_1,U_2$ such that 
\begin{itemize}
    \item every vertex $u\in U_1$ has at least $t_2$ neighbours in $U_2$, and
     \item fevery vertex $u\in U_2$ has at least $t_1$ neighbours in $U_1$. 
\end{itemize}
Then, for any $i\in [2]$, every $t\in V_i$ and $u\in U_i$, there is a copy of $T$ in $G$ where $t$ is copied to $u$, $V_1$ is copied to $U_1$ and $V_2$ is copied to $U_2$.    
\end{lemma}
We will need a well-known result of Koml\'os, S\'ark\"ozy and Szemer\'edi~\cite{KSS} which gives an optimal Dirac-type result for $n$-vertex trees with maximum degree $o(n/\log n)$.
\begin{theorem}\label{thm:KSS} Let $1/n\ll c\ll \delta\le 1$. If $G$ is an $n$-vertex graph with $\delta(G)\ge (1+\delta)n/2$, then $G$ contains a copy of every $n$-vertex tree $T$ with $\Delta(T)\le cn/\log n$.
\end{theorem}

 \subsection{Matchings and cycles in graphs}
We need the following standard results for finding cycles and matchings under degree conditions. 
\begin{lemma}[\cite{Dirac52}]Let $n\ge 3$ and let $G$ be an $n$-vertex graph with $\delta(G)\ge n/2$. Then, $G$ is Hamiltonian.\end{lemma}
\begin{lemma}[\cite{Bondy1971}]\label{lem:pancyclic}
    Let $n \ge 3$ and let $G$ be an $n$-vertex graph with $\delta(G) > n/2$. Then, $G$ is connected and non-bipartite. 
\end{lemma}
\begin{lemma}[\cite{Dirac52}]\label{lem:Dirac}
    Let $n >k \ge 2$. Each 2-connected n-vertex graph of minimum degree at least $k$ has a cycle of length at least $\min\{n, 2k\}$. In particular, each 2-connected n-vertex graph of minimum degree more than $k$ has a matching on at least $\min\{n, 2k\}$ vertices. 
\end{lemma}
%\begin{lemma}[Generalised Hall's theorem]Let $G$ be a bipartite graph with  parts $A$ and $B$, and let $(f_a)_{a\in A}$ be a tuple of non-negative integers indexed by elements of $A$. Suppose that $|N(S)|\ge \sum_{s\in S}f_s$ for all $S\subset A$. Then, there is a collection of vertex-disjoint stars $(S_a)_{a\in A}$ in $G$ where $S_a$ is centred at $a\in A$ and has exactly $f_a$ leaves.\end{lemma}
\subsection{Szemer\'edi's regularity lemma}
Let $G$ be a graph and let $A,B\subset V(G)$ be disjoint subsets. The density of the pair $(A,B)$ is defined as 
\[d(A,B)=\dfrac{e(A,B)}{|A||B|}.\]
We say that the pair $(A, B)$ is \textit{$\eps$-regular} if for all $A' \subset A$ and $B' \subset B$ such that $|A'| \ge \epsilon|A|$ and $|B'| \ge \epsilon|B|$, we have $|d(A, B) - d(A', B')| < \eps$, and say that $(A,B)$ is $(\eps,d)$-regular if, moreover, $d(A,B)\ge d$. Also, we say the graph $G[A,B]$ is $\varepsilon$-regular if $(A,B)$ is an $\varepsilon$-regular pair in $G$, and similarly for $(\eps,d)$-regular. Say a subset $X\subseteq A$ is \textit{$\eps$-significant} if $|X|> \eps |A|$, and say that a vertex $x\in A$ is  \textit{$\varepsilon$-typical} with respect to a significant set $Y\subseteq B$ if $d(x,Y)\ge(d(A,B)-\varepsilon)|Y|$. We simply write \textit{regular}, \textit{significant} or \textit{typical} whenever $\eps$ is clear from the context. The following result is a well-known fact about regular pairs. 
\begin{lemma}  \label{fact:split_regular_pairs}  Let $(A,B)$ be an $\eps$-regular pair with density $d$. 
  \begin{enumerate}[label=\upshape{(\roman{enumi})}]
    \item For each $\eps$-significant set $B' \subseteq B$, all but at most $\eps|A|$ vertices from $A$ are $\eps$-typical to~$B'$.
    \item For all $\alpha>\varepsilon$ and every $\alpha$-significant sets $A' \subseteq A$ and $B'\subseteq B$,
      $(A',B')$ is~$\frac{2\eps}{\alpha}$-regular with density $d \pm \eps$.
  \end{enumerate}
\end{lemma}

We will use the following coloured version of Szemer\'edi's regularity lemma.

\begin{theorem}\label{thm:ColRegLem}Let $1/K_0\ll 1/k_0\ll \varepsilon,d\le 1$ and let $n\ge k_0$. Suppose $G$ is a blue/red coloured $n$-vertex graph. Then, for some $k_0 \le k \le K_0$, there is a partition $V(G)=V_0\cup V_1\cup\ldots\cup V_k$ and a subgraph $G' \subseteq G$ such that $V(G')=V(G)\setminus V_0$ and the following properties hold.  
\begin{enumerate}[label=\upshape\textbf{R\arabic{enumi}}]
    \item\label{item:(R1)} $|V_0| \le \epsilon n$. 
    \item\label{item:(R2)} $|V_1|=\ldots =|V_k|\le \ceil{\epsilon n}$.
    \item\label{item:(R3)}$d_{G'}(v) > d_G(v) -(2d + \epsilon)n$ for all $v \in V(G)$.
    \item\label{item:(R4)} $V_i$ is an independent set in $G'$ for all $i \in [k]$.
    \item\label{item:(R5)} For all $1 \le i < j \le k$, both $\red{G'}[V_i, V_j]$ and $\blue{G'}[V_i,V_j]$ are $\epsilon$-regular with density either 0 or at least $d$.
\end{enumerate}
\end{theorem}
For a graph $G$, the partition $V(G)=V_0\cup V_1\cup\ldots\cup V_k$ given by Theorem~\ref{thm:ColRegLem} is referred to as an $(\eps,d)$-regular partition of $G$ and $G'$ is the cleaned graph. As usual, one can define the $(\eps,d)$-reduced graph $R$ which has vertex set $[k]$ and, for distinct $i,j\in [k]$, $ij$ is an edge with colour $*\in\{\text{blue, red}\}$ in $R$ if and only if $(V_i,V_j)$ is $(\eps,d)$-regular with density at least $d$ in $G'$.  We remark that, though not explicitly mentioned here, we can require that the partition given in Theorem~\ref{thm:ColRegLem} is a refinement of any given partition of the host graph. In particular, if the host graph is bipartite, we can ensure that each part $V_1,\ldots, V_k$ is completely contained in one of the partition classes. Moreover, note that if we pretend that both colours are the same, then Theorem~\ref{thm:ColRegLem} is just the classical degree version of the regularity lemma.  

The following standard result states that minimum degree conditions in $G$ translate to minimum degree conditions in the reduced graph.

\begin{lemma}\label{fact:H_min_deg}
    If $G$ is a graph with $\delta(G) \ge \alpha |G|$ and $R$ is an $(\eps, d)$ reduced graph of $G$, then $\delta(R) \ge (\alpha - (2d + 2\eps))|R|$. 
\end{lemma}
%\begin{proof}
%    Suppose that $\delta(H) < (\alpha - (2d + 2\eps))|H|$. Then there is some cluster $V_i$ such that $d_H(V_i) < (\alpha - (2d + 2\eps))|H|$. For each $v \in V_i$, we then have $d_{G'}(v) < (\alpha - (2d + 2\eps))|G'| + |V_0| \le (\alpha - (2d + \eps))|G|$, a contradiction, as by \eqref{item:(R3)}, $\delta(G') \ge (\alpha - (2d +\eps))|G|$. 
%\end{proof}
\section{Embedding  bounded degree trees  using regularity}\label{sec:regularity}
In this section, we use the regularity method to embed trees, inspired by the \textit{connected-matching method} introduced by \L uczak \cite{luczak1999r} in the 90s. Although this method was originally used to embed cycles, it has been adapted by several groups of authors to embed bounded-degree trees in dense graphs/hypergraphs~(see, e.g., \cite{Besomi2019,bottcher2012tripartite,chen2025embedding,pokrovskiy2025embedding}). Some of the results presented in this section are not new, but we have decided to present them here in a unified manner for the sake of completeness. 

The next result says that we can embed a small tree along a path in some reduced graph; this will be key while using the connected-matching method.
\begin{lemma}\label{lemma:embedding:walk}Let $1/m\ll \alpha\ll \beta\ll \eps\ll d\le 1$ and let $T$ be a tree with $\beta m$ vertices and $\Delta(T)\le m ^{\alpha}$. Let $G$ be a graph that contains sets $V_0,V_1,\ldots, V_t$  such that $|V_i|=m$ for all $i\in [t]$, and $(V_{i-1},V_{i})$ is $(\eps,d)$-regular for each $i\in [t]$. Suppose $T$ is rooted at some vertex $v_0$, and $u \in V_0$ is typical. Then there is an embedding $\psi:T\to G$ such that the following holds.
\begin{enumerate}[label=\upshape{(\roman{enumi})}]
    \item  We have $\psi (v_0) = u$, and for each  $0\le i<t$ and $v\in V(T)$ at distance $i$ from $v_0$, $\psi(v)\in V_i$ and $\psi(v)$ is $\eps$-typical to $V_{i + 1}$.
    \item For every vertex $v\in V(T)$ at distance at least $t$ from $v_0$, $\psi(v)\in V_{t-1}\cup V_t$ and if $\psi(v) \in V_{t-i}$, for $i \in \{0,1\}$, then $\psi(v)$ is $\eps$-typical to $V_{t-(1-i)}$.
\end{enumerate}
\end{lemma}
\begin{proof}
    We construct the embedding $\psi : T \to G$ by induction on the level of the vertices. We define a sequence $j_1,\ldots, j_h$, where $h$ is the height of the tree. For $1 \le i < t$, set $j_i = i$ and, for $t \le i  \le h$, set $j_i = t - (i-t \mod 2)$. At step $i$, we embed the vertices of level $i$ to vertices of $V_{j_i}$ which are typical to $V_{j_{i+1}}$. For the root we assign $\psi(v_0) = u$, so suppose we have an embedding $\psi_{i-1}$ of the vertices up to level $i-1$ for some $i \in [h]$. Now for each vertex $v$ of level $i+1$, the parent of $v$ has at least $(d-\epsilon)m$ neighbours in $V_{j_i}$, of which at most $\epsilon m$ are not typical to $V_{i+1}$ and at most $\beta m$ have been covered, leaving at least $(d-2\eps - \beta)m > (d-3\eps)m$ vertices to embed each $v$. As the total number of vertices in level $i$ is at most $m^{i\alpha}\ll (d-3\eps)m$, there is enough room to embed each vertex of level $i$ to a unique vertex. 
\end{proof}
The next proposition is the core result of this section and (informally) says that if we have some sort of connected structure in a reduced graph $R$ where we want to embed a given tree $T$, then it is enough to find an assignment of $T$ to $R$ so that no cluster is overused.   
\begin{proposition}\label{prop:main embedding}Let $1/n\ll 1/m\ll \alpha\ll \beta\ll 1/k\ll \varepsilon\ll d\ll \delta \le 1$. Let $T$ be an $n$-vertex tree with  $\Delta(T)\le n^{\alpha}$ and bipartition $S_1\cup S_2$, and suppose $T$ has a $\beta$-decomposition $T_1,\ldots, T_t$ as given by Lemma~\ref{beta_decomp_v2}. Let $G$ be a graph with at most $4n$ vertices that has a partition $V(G)=V_0\cup V_1\cup\ldots \cup V_k$, and let $R$ be a graph with vertex set $[k]$ such that if $ij\in E(R)$, then $(V_i,V_j)$ is $(\eps,d)$-regular. Suppose that $|V_i|=m$ for all $i\in [k]$, $R$ is connected, and there is a function $\sigma:V(T)\to V(R)$ such that the following properties hold.
\stepcounter{propcounter}
\begin{enumerate}[label=\upshape{\textbf{\Alph{propcounter}\arabic{enumi}}}]
    \item\label{prop:emb:1} $|\sigma ^{-1}(i)|\le (1-3\delta)m$ for each $i\in [k]$.
    \item\label{prop:emb:2} For every $i\in [t]$ there is an edge $\sigma_{i,1}\sigma_{i,2}\in E(R)$ such that $\sigma(V(T_i))=\{\sigma_{i,1},\sigma_{i,2}\}$ and $\sigma(V(T_i)\cap S_x)=\sigma_{i,x}$ for $x\in [2]$.
    \item\label{prop:emb:3} If $R$ is bipartite, then $\sigma$ preserves the bipartition classes of $T$.
\end{enumerate}
Then, $G$ contains a copy of $T$. Furthermore, given any $v_0 \in T$ and any typical $u \in V_q$ for some $q$ (where  $q$ is in the same bipartition class as $\sigma(v_0)$ if $R$ is bipartite), $G$ contains a copy of $T$ with $v_0$ embedded to $u$.
\end{proposition}
\begin{proof}  First, for each $i\in [k]$, we arbitrarily partition $V_i=V_{i,c}\cup V_{i,w}$ where $|V_{i,c}|=(1-2\delta)m$ and $|V_{i,w}|=2\delta m$. We order the pieces $T_1,\ldots, T_t$ in the following way. We let $v_0$ be the root $r_1$ of $T_1$. For all $i \in [t] \setminus {1}$, we let $r_i$ be the vertex closest to $v_0$ in $T$, and order the pieces so that for each $1<i\le t$ there is some $1\le j<i$ such that the parent of $r_i$ is a leaf of $T_j$.

 We use $\sigma:V(T)\to V(R)$ as a guide to produce an embedding of $\psi:T\to G$ as follows. For each $s\in [t]$ in turn, we will have a partial embedding $\psi_s:T_1\cup\ldots \cup T_s\to G$ so that the following properties hold. \stepcounter{propcounter}
\begin{enumerate}[label=\upshape{\textbf{\Alph{propcounter}\arabic{enumi}}}]
    \item\label{embedding:1.1} If $v\in V(T_i) \cap S_x$ is at distance at least $k$ from $r_i$, then $\psi_s(v)\in V_{\sigma_{i,x},c}$ and $\psi_s(v)$ is $\eps/\delta$-typical to $V_{\sigma_{i,3-x}}$.
    \item\label{embedding:1.2} If $v\in V(T_i)$ is at distance less than $k$ from $r_i$, then $\psi_s(v)\in \cup_{i\in [k]}V_{i,w}$ is $\eps/\delta$-typical to $V_{j,w}$ for some $j\in [k]$. 
\end{enumerate}

 For $s=1$, we embed $v_0$ to $u$. Let $\sigma_{i, x}$ be tha value of $\sigma(v)$ for each $v$ at distance $k$ from $v_0$. As $R$ is connected, there is a path ${j_0},\ldots, {j_\ell}$ such that $j_0=q$ and $j_\ell=\sigma_{i, x}$ for some $\ell \in [k-1]$. We note that this path always exists and, if $R$ is bipartite, then it preserves the bipartition classes of $T$ because of~\ref{prop:emb:2} and~\ref{prop:emb:3}. We note that by Lemma~\ref{fact:split_regular_pairs} each two consecutive sets in the sequence $V_{j_0, w}, \ldots, V_{j_\ell-1,w}, V_{j_\ell,c}$ are $\epsilon/d$ regular with density $d-\epsilon$, and thus we embed along this sequence using \Cref{lemma:embedding:walk}. 
 
 Now assume that, for some $1\le s<t$, we have a partial embedding $\psi:T_1\cup\ldots\cup T_s\to G$ so that both~\ref{embedding:1.1} and~\ref{embedding:1.2} hold, and let us try to extend $\psi$.  By definition of the ordering of $T_1,\ldots, T_t$, there is some $j\in [s]$ and a vertex $v_j\in V(T_j)$ such that $v_jr_{s+1}\in E(T)$. Note that either~\ref{embedding:1.1} or~\ref{embedding:1.2} imply that there is some $s_0\in [k]$ such that $\psi(v_j)$ is typical to $V_{s_0,w}$. For all $v \in T_{s+1}$ at distance $k$ from $r_{s+1}$, let $\psi(v) = \sigma_{s+1, x}$. As $R$ is connected, there is a path ${j_0},\ldots, {j_\ell}$ such that $j_0=s_0$ and $j_\ell=\sigma_{s+1, x}$ for some $\ell \in [k-1]$. We note that this path always exists and, if $R$ is bipartite, then it preserves the bipartition classes of $T$ because of~\ref{prop:emb:2} and~\ref{prop:emb:3}. 

Using $\Delta(T)\le n^\alpha$  and \ref{embedding:1.1}, we deduce that the number of vertices mapped to $\cup_{i\in [k]}V_{i,w}$ is at most
\[t\cdot (1+\Delta(T)+\ldots +\Delta(T)^k)\le 2t\Delta(T)^{k+1}\le 2tn^{\alpha(k+1)}=o(n),\]
as $\alpha\ll 1/k$ and $t< 1/\beta$. Moreover, note that~\ref{prop:emb:1} implies the number of unoccupied vertices in each $V_{i,c}$ is at least $(1-2\delta)m-(1-3\delta)m= \delta m$. Our goal is to apply Lemma~\ref{lemma:embedding:walk} with $V_i'$ as $V_i$, where $V_i'$ is given as follows: 
\begin{equation*}
V_i' = 
    \begin{cases}
        \text{$V_{j_i, w} \setminus \im(\psi_s)$} & \text{for all $i < \ell$,}\\
        \text{$V_{\sigma_{s+1, x_i},w}\setminus \im(\psi_s)$} & \text{for all $\ell \le i < k$, and}\\
        \text{$V_{\sigma_{s+1, x_i},c}\setminus \im(\psi_s)$} & \text{for all $k\le i$,}\\
    \end{cases}
\end{equation*}
where for each $i$, we define $x_i = x+((l-i) \mod 2)$. To apply Lemma~\ref{lemma:embedding:walk}, we again note that by Lemma~\ref{fact:split_regular_pairs} we have $(V_{i-1}',V_{i}')$ are $(\frac{\eps}{\delta},d-\eps)$-regular for each $i\in [s]$. This suffices to embed $T_{s+1}$ along this walk so that both \ref{embedding:1.1} and~\ref{embedding:1.2} hold. 
\end{proof}

\subsection{Embedding lemmas}
In this section, we apply Proposition~\ref{prop:main embedding} to prove a series of embedding results that will later be used in the proof of Theorems~\ref{thm:stability} and~\ref{thm:stability:unbounded}. 
\begin{lemma}\label{lem:emb1}Let $1/n\ll 1/m\ll \alpha\ll 1/k\ll \varepsilon\ll d\ll\delta\le 1$. Let $G$ be a graph with at most $4n$ vertices that has a partition $V(G)=V_1\cup\ldots \cup V_k$, and let $R$ be a graph with vertex set $[k]$ such that if $ij\in E(R)$, then $(V_i,V_j)$ is $(\eps,d)$-regular. Suppose there are disjoint sets $I, J\subset [k]$ with $|I|=|J|$ such that 
\stepcounter{propcounter}
\begin{enumerate}[label=\upshape{\textbf{\Alph{propcounter}\arabic{enumi}}}]
\item $|V_i|=m$ for all $i\in [k]$,
\item\label{lem:emb1:2}$|\cup_{i\in I}V_i|=|\cup_{i\in J}V_i|=(1+2\delta)n/2$,
\item $R$ is connected and non-bipartite, and
\item\label{lem:emb1:4} In $R$, there is a perfect matching between $I$ and $J$. 
\end{enumerate}
Then, $G$ contains a copy of every $n$-vertex tree $T$ with $\Delta(T)\le n^\alpha$.
\end{lemma}
\begin{proof}
Let $S_1\cup S_2$ be the bipartition classes of $T$. We introduce two auxiliary constants $\beta,\mu$ so that $\alpha\ll \beta\ll 1/k\ll \mu\ll \delta$ and apply Lemma~\ref{beta_decomp_v2} to get a $\beta$-decomposition $T=T_1\cup\ldots \cup T_t$ with $t=O(\beta^{-2})$. Our goal is then to find an assignment $\sigma:V(T)\to V(R)$ so that \ref{prop:emb:1} and~\ref{prop:emb:2} hold. To do so, we will run a random process to assign the pieces $T_1,\ldots, T_t$ to the edges of the perfect matching between $I$ and $J$, as follows. 

For $i\in [t]$, let $c(i)\in I$  and  $Y_i\in\{0,1\}$ be chosen uniformly at random (all choices made independently). Let $i\in I$ be fixed and, for $s\in [t]$, let $X_{s,i}=1$ whenever $c(s)=i$ and $0$ otherwise. Let 
\[f(X_{1,i},\ldots, X_{t,i},Y_1,\ldots, Y_t)=\sum_{s\in [t]}(Y_s\cdot |V(T_s)\cap S_1|+(1-Y_s)\cdot|V(T_s)\cap S_2|)\cdot X_{s,i}\] 
and note that 
\[\mathbb E[f(X_{1,i},\ldots, X_{t,i},Y_1,\ldots, Y_t)]=\frac{1}{2|I|}\sum_{j\in [t]}|T_j|=\frac{n}{2|I|}.\]
Recalling that $|T_i|\le \beta n$ for each $i\in [t]$, we have $\sum_{i\in [t]}|T_i|^2\le \beta n\sum_{i\in [t]}|T_i|\le \beta n^2$. Then, McDiarmid's inequality (Lemma~\ref{lemma:mcdiarmid}) gives 
\[\mathbb P\left(\left|f(X_{1,i},\ldots, X_{t,i}, Y_1,\ldots, Y_t)-\tfrac{n}{2|I|}\right|>\mu m\right)\le 2\exp\left(\tfrac{-2\mu^2m^2}{2\sum_{i\in [t]}|T_i|^2}\right )\le 2\exp (-\mu^2/\beta k^2),\]
where we used that $km\ge m(|I|+|J|)\ge n$ by~\ref{lem:emb1:2}. Then, taking a union bound over all $i\in I$, 
\begin{equation}\label{eq:event:1}\left|\sum_{s\in c^{-1}(i)}\big(Y_s\cdot |V(T_s)\cap S_1|+(1-Y_s)\cdot |V(T_s)\cap S_2|)-\frac{n}{2|I|}\right|\le \mu m 
\end{equation}
holds for all $i\in I$ with probability at least $1-2k\exp(-\mu^2/\beta k^2)>0$ as $\beta\ll\mu,1/k$.

Let $\sigma:V(T)\to V(R)$ be defined as follows. For each $i\in I$, let $M(i)\in J$ be the matched vertex given by~\ref{lem:emb1:4}.  For each $s\in [t]$, if $Y_s=1$ then set $\sigma(v)={c(s)}$ for every $v\in V(T_s)\cap S_1$ and $\sigma(v)={M(c(s))}$ for every $v\in V(T)\cap S_2$, and otherwise set $\sigma(v)={M(c(s))}$ for every $v\in V(T_s)\cap S_1$ and $\sigma(v)={c(s)}$ for every $v\in V(T)\cap S_2$. Noting that 
\[|\sigma^{-1}(i)|= \sum_{s\in c^{-1}(i)}\big(Y_s\cdot |V(T_s)\cap S_1|+(1-Y_s)\cdot |V(T_s)\cap S_2|)\overset{\eqref{eq:event:1}}{\le }\frac{n}{2|I|}+\mu m\le (1-\delta)m,\]
as~\ref{lem:emb1:2} implies $n/2|I|= (1+2\delta)^{-1}m$, there is  a choice of $\sigma$ for which \ref{prop:emb:1} and~\ref{prop:emb:2} hold and thus Proposition~\ref{prop:main embedding} implies $G$ contains a copy of $T$.
\end{proof}
\begin{lemma}\label{lem:emb2}Let $1/n\ll 1/m\ll \alpha \ll 1/k\ll\eps\ll d\ll \delta\le 1$. Let $G$ be a graph on at most $4n$ vertices that has a partition $V(G)=V_1\cup \ldots\cup V_k$, and let $R$ be the graph with vertex set $[k]$ such that if $ij\in E(R)$, then $(V_i,V_j)$ is $(\eps,d)$-regular. Suppose there are disjoint sets $I,J\subset [k]$ such that the following properties hold.
\stepcounter{propcounter}
\begin{enumerate}[label=\emph{\textbf{\Alph{propcounter}\arabic{enumi}}}]
\item $|V_i|=m$ for all $i\in [k]$. 
\item $|\cup_{i\in I}V_i|=|\cup_{i\in J}V_j|=(1+100\delta)n/2$.
\item $R$ is connected and there is a perfect matching between $I$ and $J$.
\end{enumerate}
Then, $G$ contains a copy of every $n$-vertex tree $T$ with $\Delta(T)\le n^{\alpha}$ and partition classes of sizes $t_1,t_2$ such that $t_1\ge t_2\ge (1/2-\delta)n$.
\end{lemma}
The proof of Lemma~\ref{lem:emb2} is essentially the same as the proof of Lemma~\ref{lem:emb1}, but letting $Y_s=1$ deterministically for all $i\in [t]$, and so we omit the proof. 
\begin{lemma}\label{lem:emb3:bipartite}Let $1/n\ll 1/m\ll \alpha\ll 1/k \ll \varepsilon\ll \delta \ll d\le 1$. Let $G$ be a graph with at most $4n$ vertices that has a partition $V(G)=V_1\cup\ldots \cup V_k$, and let $R$ be a graph with vertex set $[k]$ such that if $ij\in E(R)$, then $(V_i,V_j)$ is $(\eps,d)$-regular. Suppose there are disjoint sets $I, J\subset [k]$ such that the following properties hold.
\stepcounter{propcounter}
\begin{enumerate}[label=\upshape{\textbf{\Alph{propcounter}\arabic{enumi}}}]
\item $|V_i|=m$ for all $i\in [k]$.
\item\label{lem:emb2:2} $|\cup_{i\in I}V_i|=(1+5\delta)t_1$.
\item $R$ is connected and bipartite with parts $I\cup J$.
\item\label{lem:emb2:4}In $R$, every $i\in I$ is adjacent to at least $(1+5\delta)\frac{t_2}{m}$ vertices $j\in J$.
\end{enumerate}
Then, $G$ contains a copy of every $n$-vertex tree $T$ with $\Delta(T)\le n^\alpha$ and bipartition classes of sizes $t_1$ and $t_2$.
\end{lemma}
\begin{proof}
    
Let $S_1\cup S_2$ be the bipartition classes of $T$ so that $|S_1|=t_1$ and $|S_2|=t_2$. We introduce two auxiliary constants $\mu,\beta$ so that $\alpha\ll \beta\ll 1/k\ll \mu\ll \delta$ and apply  Lemma~\ref{beta_decomp_v2} to get a $\beta$-decomposition $T=T_1\cup\ldots \cup T_t$. Again, our goal is to find an assignment $\sigma:V(T)\to V(R)$ so that \ref{prop:emb:1} and~\ref{prop:emb:2} hold.

For each $s\in [t]$, pick $i(s)\in I$ uniformly at random and then pick $j(s)\in J$ uniformly at random among the neighbours of $i(s)$ in $R$ (all choices are made independently). Let $a\in I$ and $b\in J$ be fixed. For each $s\in [t]$, let $X_{s,a}=1$ if $i(s)=a$ and $0$ otherwise and let $Y_{s,b}=1$ whenever $j(s)=b$ and 0 otherwise. Let 
\[f(X_{1,a},\ldots, X_{t,a})=\sum_{s\in [t]}X_{s,a}\cdot |V(T_s)\cap S_1|\quad\text{and}\quad g(Y_{1,b},\ldots, Y_{t,b})=\sum_{s\in [t]}Y_{s,b}\cdot |V(T_s)\cap S_2|.\] 
First, we clearly have
\[\mathbb E[f(X_{1,a},\ldots, X_{t,a})]=\frac{1}{|I|}\sum_{s\in [t]}|V(T_s)\cap S_1|=\frac{t_1}{|I|}\le \frac{m}{1+5 \delta} \le (1-4\delta)m,\]
as~\ref{lem:emb2:2} implies $m|I|=(1+5\delta)t_1$.
Secondly, using~\ref{lem:emb2:4} and conditioning we get
\[\mathbb P(Y_{s,b}=1)=\sum_{a\in N_R(b)}\mathbb P(Y_{s,b}=1|X_{s,a}=1)\mathbb P(X_{s,a}=1)\le \frac{m}{(1+5\delta)t_2} \] 
and thus
\[\mathbb E[g(Y_{1,b},\ldots, Y_{t,b})]=\sum_{s\in [t]}|V(T_s)\cap S_2|\cdot \mathbb P(Y_{s,b}=1)\le \frac{m}{1+5\delta}\le (1-4\delta)m\]
Also, using that $|T_i|\le \beta n$ for each $i\in [t]$, we get $\sum_{i\in [t]}|V(T_i)\cap S_j|^2\le \beta n\sum_{i\in [t]}|V(T_i)\cap S_j|\le \beta nt_j$ for $j\in [2]$. Then, McDiarmid's inequality (Lemma~\ref{lemma:mcdiarmid}) gives 
\begin{equation}\label{eq:emb2:1}\mathbb P\left(f(X_{1,a},\ldots, X_{t,a})>(1-4\delta)m+\mu m\right)\le 2\exp\left(\tfrac{-2\mu^2m^2}{\sum_{s\in [t]}|V(T_s)\cap S_1|^2}\right )\le 2\exp (-2\mu^2m^2/\beta nt_1)\end{equation}
and 
\begin{equation}\label{eq:emb2:2}\mathbb P\left(g(X_{1,b},\ldots, X_{t,b})>(1-4\delta)m+\mu m\right)\le 2\exp\left(\tfrac{-2\mu^2m^2}{\sum_{s\in [t]}|V(T_s)\cap S_2|^2}\right )\le 2\exp (-2\mu^2m^2/\beta nt_2).\end{equation}
Let $\sigma :V(T)\to I\cup J$ be defined as follows.  For every $s\in [t]$ and $v\in V(T_s)$, we let $\sigma(v)=i(s)$ if $v\in V(T_s)\cap S_1$ and $\sigma(v)=j(s)$ for if $v\in V(T_s)\cap S_2$. Noting that~\ref{lem:emb2:2} and~\ref{lem:emb2:4} imply $mk\ge m|I\cup J|\ge n$, by a union bound we have that, for every $i\in [k]$, 
\[\left|\sigma^{-1}(V_i)\right|\overset{\eqref{eq:emb2:1},\eqref{eq:emb2:2}}{\le} (1-4\delta)m+\mu m\le (1-3\delta )m\]
with probability at least $1-2k\exp(-2\mu^2m^2/\beta n^2)\ge 1-2k\exp(-2\mu^2/\beta k^2)>0$, as $\beta\ll\mu,1/k\ll1$. Therefore, there is a choice of $\sigma:V(T)\to V(R)$ for which \ref{prop:emb:1}--\ref{prop:emb:3} hold.\end{proof}

%\begin{corollary}\label{cor:balanced_bipartite_all_trees}

%Given $\Delta \in \NN$ and $1/\Delta \gg \delta > 0$, there is some $n_0 \in N$ such that the following holds for all $n \ge n_0$. Let $\epsilon, d$ be such that $0 < \epsilon \ll d \ll \delta$. and let $G$ be a 2-edge-coloured graph with $|G|= 2n-1$, $\delta(G) \ge 3/4|G|$. Let $H = R_H \cup B_H$ be a $(\epsilon, d)$-reduced 2-edge coloured graph of $G$ with $|H| = k$.  \\
%Suppose $H$ contains a monochromatic bipartite component $C = (X, Y)$ with $|C| \ge (1 - \delta)k, |X| \ge |Y| \ge (1/2-10\delta)k$ and minimum degree (in the colour of the component) at least $(\frac 14 - 10\delta)k$. Then $G$ contains all trees on $n$ vertices with maximum degree $\Delta$. 
%\end{corollary}
\begin{lemma}\label{cor:balanced_bipartite_all_trees} Let $1/n\ll 1/m\ll \alpha\ll 1/k\ll \eps\ll d\ll \delta\ll 1$. Let $G$ be a graph with at most $4n$ vertices that has a partition $V(G)=V_1\cup\ldots\cup V_k$, and let $R$ be a graph with vertex set $[k]$ such that if $ij\in E(R)$, then $(V_i,V_j)$ is $(\eps,d)$-regular. Suppose there are disjoint sets $I,J\subset [k]$ such that the following properties hold. 
\stepcounter{propcounter}
\begin{enumerate}[label=\upshape{\textbf{\Alph{propcounter}\arabic{enumi}}}]
    \item\label{emb3:1} $|V_i|=m$ for all $i\in [k]$.
    \item\label{emb3:2} $|\cup_{i\in I}V_i|\ge |\cup_{i\in J}V_i|\ge (1-\delta)n$.
    \item\label{emb3:3} $R$ is connected and bipartite with parts $I\cup J$.
    \item\label{emb3:4} In $R$, every $i\in I$ is adjacent to at least $(1-\delta)|J|/2$ vertices in $J$, and every $j\in J$ is adjacent to at least $(1-\delta)|I|/2$ vertices in $I$.
\end{enumerate}
Then, $G$ contains a copy of every $n$-vertex tree $T$ with $\Delta(T)\le n^{\alpha}$ and bipartition classes of sizes at most $(1-5\delta )n$. 
\end{lemma}

\begin{proof}Let $S_1\cup S_2$ be the bipartition classes of $T$ and assume  $(1-5\delta)n\ge |S_1|\ge |S_2|$. We introduce two auxiliary constants $\mu,\beta$ so that $\alpha\ll \beta\ll 1/k\ll \mu\ll \delta$ and apply Lemma~\ref{beta_decomp_v2} to get a $\beta$-decomposition $T=T_1\cup\ldots \cup T_t$. Again, our goal is to find an assignment $\sigma:V(T)\to V(R)$ so that \ref{prop:emb:1} and~\ref{prop:emb:2} hold. Note that if $|S_2|\le (1-100\delta)n/2$, then every $i\in I$ is adjacent to at least $(1-\delta)|J|/2\ge (1-\delta)^2n/2m\ge (1+5\delta)|S_2|/m$ clusters in $J$, and thus Lemma~\ref{lem:emb3:bipartite} implies $T$ is contained in $G$. Therefore, we may further assume that $(1-100\delta)\frac{n}{2}\le |S_2|\le |S_1|\le (1+200\delta)\frac{n}{2}.$

For $p=800\delta$, let $K\subset [t]$ be a random subset such that $\mathbb P(i\in K)=p$ and all choices are independent. For $j\in [2]$, we have $\mathbb E[|\cup_{i\in K}V(T_i)\cap S_j|]=p|S_j|$, Chernoff's bound (Lemma~\ref{lemma:chernoff}) gives that 
\[\mathbb P\left(\big||\cup_{i\in K}V(T_i)\cap S_j|-p|S_j|\big|>p|S_j|/2\right)\le 2\exp(-p^2|S_j|/12).\]
Fix a choice of $K\subset [t]$ such that $p|S_j|/2\le |\cup_{i\in K}V(T_i)\cap S_j|\le 3p|S_j|/2$ for $j\in [2]$. We will assign $(T_i)_{i\in [t]}$ in two rounds. In the first round, for each $s\in [t]\setminus K$, pick $j(s)\in J$ uniformly at random and then pick $i(s)\in I$ uniformly at random among the neighbours of $j(s)$ in $R$ (all choices are made independently). Let $a\in J$ and $b\in I$ be fixed. For each $s\in [t]\setminus K$, let $X_{s,a}=1$ if $j(s)=a$ and $0$ otherwise and let $Y_{s,b}=1$ whenever $i(s)=b$ and 0 otherwise. Let 
\[f(X_{1,a},\ldots, X_{t,a})=\sum_{s\in [t]\setminus K}X_{s,a}\cdot |V(T_s)\cap S_2|\quad\text{and}\quad g(Y_{1,b},\ldots, Y_{t,b})=\sum_{s\in [t]\setminus K}Y_{s,b}\cdot |V(T_s)\cap S_1|.\]
First, we clearly have $\mathbb E[f(X_{1,a},\ldots, X_{t,a})]=\frac{1}{|J|}\sum_{s\in [t]\setminus K}|V(T_s)\cap S_2|\le \frac{|S_2|}{|J|}\le \frac{m}{2(1-\delta)},$ 
as~\ref{emb3:2} implies $m|J|\ge (1-\delta)n\ge (1-\delta)2|S_2|$.
Secondly, using~\ref{emb3:4} and conditioning we get $\mathbb P(Y_{s,b}=1)\le\frac{2}{(1-\delta)|I|}$
and thus
\[\mathbb E[g(Y_{1,b},\ldots, Y_{t,b})]\le \frac{2}{(1-\delta)|I|}\cdot ((1+200\delta)n/2-|\cup_{i\in K}V(T_i)\cap S_1|)\le (1-5\delta)m,\]
where we used that $|\cup_{i\in K}V(T_i)\cap S_1|\ge p|S_1|/2\ge 200\delta n$ and $|I|m\ge (1-\delta)n$. Similarly to the proofs of Lemmas~\ref{lem:emb1} and~\ref{lem:emb2}, McDiarmid's inequality (Lemma~\ref{lemma:mcdiarmid}) gives that with positive probability,
\[f(X_{1,a},\ldots, X_{t,a})\le 3m/4\quad\text{and}\quad g(Y_{1,b},\ldots, Y_{t,b})\le (1-4\delta)m\]
for all $a\in J$ and $b\in I$. Let $\sigma:V(T)\to I\cup J$ so that, for every $s\in [t]\setminus K$ and $v\in V(T_s)$, $\sigma(v)=i(s)$ if $v\in V(T_s)\cap S_1$ and $\sigma(v)=j(s)$ if $v\in V(T_s)\cap S_2$. It only remains to assign the vertices in $\cup_{s\in K}T_s$. Let $I'\subset I$ be indices such that $|\sigma^{-1}(i)|\le 3m/4$. Then we have $|I\setminus I'|\cdot 3m/4\le (1+200\delta)n/2$ and hence
\begin{equation}\label{eq:I'}|I'|m\ge 3m|I|/4-(1+200\delta)n/2\ge n/8.  
\end{equation}
Now, we greedily assign the remaining subtrees as follows. Label $K=\{s_1,\ldots, s_{|K|}\}$. For $\ell\in [|K|]$ in turn, pick some $i\in I'$ such that $|\sigma^{-1}(i)|\le (1-4\delta)m$ and some $j\in N_R(i)$ with $|\sigma^{-1}(j)|\le (1-4\delta)m$ and set $\sigma(v)=i$ for every $v\in V(T_{s_\ell})\cap S_1$ and $\sigma(v)=j$ for every $v\in V(T_{s_\ell})\cap S_2$. This is always possible as \eqref{eq:I'} implies there is always a choice for $i\in I'$, and if there is no choice for $j$ then we would have assigned 
\[d_R(i)\cdot (m/4-4\delta m)\ge |J|m/10\ge n/20\]
vertices,  contradicting that $|\cup_{s\in K}T_s|\le 3pn/2\le 2000\delta n$. Therefore, we have found $\sigma:V(T)\to V(R)$ for which \ref{prop:emb:1}--\ref{prop:emb:3} hold.
\end{proof}

\begin{lemma}\label{lemma:bipartite:almost:spanning} Let $1/n\ll \alpha\ll\mu,\delta\le 1$ and let $\gamma\in (0,1)$ with $\mu\ll\gamma$. Let $G$ be a bipartite graph with parts $U_1,U_2$ such that $n\ge |U_i|\ge (1-\mu)n$ for $i\in [2]$, and suppose that for every $i\in [2]$ and $u\in U_i$, $d(u)\ge (1+\delta)|U_{3-i}|/2$. Let $T$ be an $n$-vertex tree with $\Delta(T)\le n^{\alpha}$ and bipartition classes of sizes at most $(1-\gamma)n$. Then, for every $t\in V(T)$ and $u_0\in V(G)$, there is a copy of $T$ in $G$ where $t$ is copied to $u_0$.
\end{lemma}
\begin{proof}
We introduce auxiliary constants $\eps,d\in (0,1)$ and $k_0,K_0\in\mathbb N$ so that $1/n\ll1/K_0\ll 1/k_0\ll\eps\ll d\ll \alpha$. For some $k_0\le k\le K_0$, Theorem~\ref{thm:ColRegLem} gives a partition $V(G)=V_0\cup V_1\cup\ldots\cup V_k$ and a subgraph $G'\subset G$ satisfying~\ref{item:(R1)}--\ref{item:(R5)}. As $G$ is bipartite, there is a partition $[k]=I_A\cup I_B$ such that $V_i\subset U_1$ for every $i\in I_1$ and $V_j\subset U_2$ for every $j\in I_2$. Moreover, by \ref{item:(R3)} and the degree conditions in $G$, we have that for every $i\in [2]$ and $u\in U_{3-i}$, $d_{G'}(u)\ge (1+\delta/2){|U_{3-i}|}/{2}$.
Let $R$ be the $(\eps,d)$-reduced graph of $G$. Note that $R$ is bipartite with bipartition $I_1\cup I_2$, and 
\begin{itemize}
    \item  $d_{R}(j)\ge (1+\delta/4){|I_{3-i}|}/{2}$ for every $i\in [2]$ and $j\in I_i$,\end{itemize}
which implies $R$ is connected. We introduce an axillary constant $\beta$ such that $\alpha \ll \beta \ll 1/k$,  and create a $\beta$-decomposition $T = T_1 \cup ... \cup T_\ell$ of $T$, where $t \in V(T_1)$. Suppose without loss of generality that $u_0\in U_1$. To start the embedding at $u_0$, we first pick $j\in [k]$ such that $V_j\subset U_2$ and 
\[d_G(u_0,V_j)\ge \frac{(1+\delta)|U_2|/2-|V_0|}{k}\ge (d-\varepsilon)|V_j|,\]
and then pick $i\in [k]$ so that $(V_i,V_j)$ is $(\eps,d)$-regular. Add $u_0$ to $V_i$ and, as in the proof of Lemma \ref{cor:balanced_bipartite_all_trees}, find an assignment $\sigma : V(T) \to [k]$ such that \ref{prop:emb:1}--\ref{prop:emb:3} hold and $T_1$ is assigned to $(V_i,V_j)$ with $\sigma(t)=u_0$. We can therefore use \Cref{prop:main embedding} to embed $T$ in $G$ with $t$ copied to $u_0$. 

\end{proof}

\section{Stability}\label{sec:stability}
As mentioned in the introduction, for the proof of Theorems~\ref{thm:main},~\ref{thm:main3} and~\ref{thm:main2} we follow a stability approach, whereby we prove that in any blue/red coloured graph we can either find a monochromatic copy of every bounded-degree tree, or the structure of our graph is close to the extremal example.  To prove such a stability result, we need the following lemma, which follows by combining Lemmas~4.1 and~4.2 from~\cite{cycles_2012}.

\begin{lemma}\label{lem:H_strcuture}
Let $1/N\ll 1/\delta\ll 1$ and let $G$ be blue/red coloured $N$-vertex graph with $\delta(G) \ge (3/4 - \delta)N$. Then, one of the following holds. 
\begin{enumerate} [label=\upshape(\roman{enumi})]
    \item\label{structure1} There is a non-bipartite component of $\red{G}$ or $\blue{G}$ which contains a matching with $(1/2 + \delta)N$ vertices. 
    \item\label{structure2} There is a bipartite component $C$ of $\red{G}$ or $\blue{G}$ such that $|C| \ge (1-5\delta)N$ and $C$ has a matching with $(1/2 + \delta)N$ vertices. 
    \item There is a set $S$ with $|S|\ge (2/3 - \delta)N$ such that either $\Delta(\red{G}[S]) \le 10 \delta N$ or $\Delta(\blue{G}[S]) \le 10 \delta N$.\label{structure3}
    \item There is a partition $V(G) = U_1 \cup U_2\cup U_3\cup U_4$ with $|U_i| \ge (1/4 - 3\delta)N$ for all $i\in [4]$, such that  $\blue{e}(U_1 \cup U_2,U_3 \cup U_4)=0$ and $\red{e}(U_1 \cup U_3,U_2 \cup U_4)=0$.\label{structure4}
\end{enumerate}
\end{lemma}
The next theorem is our stability result for bounded-degree trees.
\begin{theorem}\label{thm:stability}Let $\Delta\ge 2$ and let $1/n\ll\mu\ll 1/\Delta$. Let $G$ be a blue/red coloured graph with at least $2n-2$ vertices and $\delta(G)\ge  \lfloor \frac{3|G|}{4} \rfloor$ and suppose $T$ is an $n$-vertex tree with $\Delta(T)\le \Delta$. Then, either $G$ contains a monochromatic copy of $T$ or, after possibly swapping colours, there is a partition $V(G)=U_0\cup U_1\cup U_2\cup U_3\cup U_4$ such that 
\begin{itemize}
    \item $|U_i|\ge (1-\mu)n/2$ for all $i\in [4]$,
    \item $\blue{e}(U_1\cup U_2,U_3\cup U_4)\le \mu n^2$, and
    \item $\red{e}(U_1\cup U_3, U_2\cup U_4)\le \mu n^2$. 
\end{itemize} 
\end{theorem}
\begin{proof} We introduce auxiliary constants $\eps,\delta,d>0$ so that $\eps\ll d\ll \delta\ll \mu$. Use Theorem~\ref{thm:ColRegLem} with parameters $\varepsilon,d$ to get a partition $V(G)=V_0\cup V_1\cup\ldots\cup V_k$ so that \ref{item:(R1)}--\ref{item:(R5)} hold for some subgraph $G'\subset G$, and let $m=|V_1|=\ldots=|V_k|$. Note that~\ref{item:(R1)} gives $m\ge (1-\eps)|G|/k\ge (1-2\eps)2n/k$. 

Let $R$ be the graph with vertex set $[k]$ where $ij\in E(R)$ if both $\blue{G'}[V_i,V_j]$  and $\red{G'}[V_i,V_j]$ are $\eps$-regular. For distinct $i,j\in [k]$, we colour  $ij$ blue if $\blue{G'}[V_i,V_j]$ is $(\eps,d)$-regular or red if $\red{G'}[V_i,V_j]$ is $(\eps,d)$-regular, breaking ties arbitrarily. Then, Lemma~\ref{fact:H_min_deg} gives
\begin{equation}\label{eq:mindeg:reduced}\delta(R)\ge (3/4-(3d+\eps)|R|\ge (3/4-\delta)|R|,
\end{equation}
as $\eps,d\ll\delta$, and thus Lemma~\ref{lem:H_strcuture} implies one of \ref{structure1}--\ref{structure4} holds. We first show that if we are in case~\ref{structure4} of Lemma~\ref{lem:H_strcuture}, then we are done.

\begin{claim}If there is a partition $V(R)=I_1\cup I_2\cup I_3\cup I_4$ with $|I_i|\ge (1/4-3\delta)k$ for all $i\in [4]$, such that $\blue{e}(I_1\cup I_2,I_3\cup I_4)=0$ and $\red{e}(I_1\cup I_3, I_2\cup I_4)=0$, then the conclusion of Theorem~\ref{thm:stability} holds.\end{claim}
\begin{proof}[Proof of claim]
For $i\in [4]$, let $U_i=\cup_{j\in I_i}V_j$, and let $U_0=V(G)\setminus (U_1\cup U_2\cup U_3\cup U_4)$. First note that $|U_0|\le \eps |G|\le \mu n$ as $\eps\ll \mu$. Secondly, for $i\in [4]$,
\[|U_i|=m|I_i|\ge (1-2\eps)\frac{2n}{k}\cdot (1/4-3\delta)k\ge (1-\mu)n/2,\]
as $\eps,\delta\ll\mu$. Lastly, by~\ref{item:(R3)} and~\ref{item:(R5)} we have 
    \[\blue{e}(U_1\cup U_2,U_3\cup U_4)\le |U_1\cup U_2|\cdot (2d+\eps)|G|\le \mu n^2,\]
    as $\eps,d\ll\mu$. Similar arguments give $\red{e}(U_1\cup U_3, U_2\cup U_4)\le |U_1\cup U_3|\cdot (2d+\eps)|G|\le \mu n^2.$
\end{proof}
The aim now will be to prove that in cases \ref{structure1},~\ref{structure2}, and~\ref{structure3} we can find a monochromatic copy of $T$. We start with cases~\ref{structure1} and~\ref{structure3} as case~\ref{structure2} in considerable more involed.\\

\noindent \textbf{Case (i):} There is a non-bipartite component $C$ of $\red{R}$ or $\blue{R}$ which contains a matching with $(1/2+\delta)k$ vertices.  \\

Without loss of generality, let us assume that $C$  is a red component. Let $I,J\subset [k]$ be the red matching in $C$ and note that 
\[|\cup_{i\in I}V_i|=|\cup_{j\in J}V_j|=(1/2+\delta)k/2\ge (1/2+\delta)(1-2\eps)n\ge (1+\delta/2)n/2,\]
provided $\eps\ll\delta$. Then, Lemma~\ref{lem:emb1} implies $\red{G}$ contains a copy of $T$.\\

\noindent\textbf{Case (iii):} There is a set $S$ with $|S|\ge (2/3-\delta)k$ such that either $\Delta(\red{R}[S])\le 10\delta k$ or $\Delta(\blue{R}[S])\le 10\delta k$.\\

Without loss of generality we assume that $\Delta(\red{R}[S])\le 10\delta k$. Using that $\eps\ll\delta\ll 1$, observe first that 
\[|\cup_{i\in S}V_i|\ge (2/3-\delta)km\ge (2/3-\delta)(1-2\eps)2n\ge (1+10\delta)n.\]
Now,~\eqref{eq:mindeg:reduced} implies 
\[\delta(\blue{R}[S])\ge (3/4-\delta)k-10\delta k-(k-|S|)=|S|-(1/4-11\delta)k>|S|/2,\]
and thus $\blue{R}[S]$ is non-bipartite and has a perfect matching. Then, as $|\cup_{i\in S}V_i|\ge (1+10\delta)n$ and $\eps\ll\delta$, $\blue{G}[S]$ contains a copy of $T$ by Lemma~\ref{lem:emb1}.\\

\noindent \textbf{Case (ii):}  There is a bipartite component $C$ of $\red{R}$ or $\blue{R}$ such that $|C| \ge (1-5\delta)k$ and $C$ has a matching with $(1/2 + \delta)k$ vertices. \\

Without loss of generality let us assume that $C$  is a blue component, and let $I\cup J$ be the bipartition classes of $C$ so that $|I|\ge |J|$. Let $V_I=\cup_{i\in I}V_i$ and $V_J=\cup_{i\in J}V_i$ and note that 
\begin{equation}
    |V_I\cup V_J| \ge (1-5\delta)km\ge (1-5\delta)(1-\eps)2n\ge(1-6\delta)|G|.
\end{equation}
\begin{claim}If $|V_I|\ge (1+100\delta)n$, then $\red{G}$ contains a copy of $T$.
\end{claim}
\begin{proof}[Proof of Claim]
As $C$ is bipartite, then $\blue{R}[I]$ and $\blue{R}[J]$ are empty, and so,  because of~\ref{item:(R1)}, we have $\Delta(\blue{G}[V_I])\le (2d+\eps)2n$ and $\Delta(\blue{G}[V_J])\le (2d+\eps)2n$. Given any vertex $x\in V_I$, we have 
\[\red{d}(x)\ge \lfloor3|G|/4\rfloor-\Delta(\blue{G}[V_I])-|V_J|-6\delta n\ge \frac{3}{4}|V_I|-\frac{|V_J|}{4}-7\delta n\ge (1/2+\delta)|V_I|,\]
where we used that $|V_J|\le 2n-|V_I|\le (1-100\delta)n$ and $\eps,d\ll\delta$. Then, as $|V_I|\ge (1+\delta)n$, Theorem~\ref{thm:KSS} implies $\red{G}$ contains a copy of $T$. 
\end{proof}
 We assume from now on that $|V_I|\le (1+100\delta)n$ from which it follows that $|V_I|\ge |V_J|\ge (1-110\delta)n$ and $|I|\le (1+110\delta )k/2$. First, we have
 \begin{equation}\label{eq:min_Deg_H[X]}
 \delta(\red{R}[I])\overset{\eqref{eq:mindeg:reduced}}{\ge} (3/4-\delta)k-(k-|I|)=|I|-(1/4-\delta)k\ge (1-200\delta)|I|/2,
 \end{equation}
and similarly for $J$ we have $\delta(\red{R}[J])\ge (1-200\delta)|J|/2$.

\begin{claim}\label{fact:structure}For $K\in \{I,J\}$, the following holds.
    \begin{enumerate} [label=\upshape{(\roman{enumi})}]
        \item If $\red{R}[K]$ is 2-connected, then $\red{R}[K]$ contains a matching on at least $(1-  400\delta)k/2$ vertices. \label{item:X2-conn}
       
        \item If $\red{R}[K]$ is not 2-connected, then there is some $i \in K$ such that $\red{R}[K\setminus\{i\}]$ has exactly two components $K_1$ and $K_2$, such that, for $i\in[2]$, $K_i$ is non-bipartite, contains a matching on at least $(1-20\delta)k/4$ vertices, and $|K_i| \le (1+200\delta)k/4$. \label{item:X1-conn}
    \end{enumerate}
\end{claim}
\begin{proof}[Proof of claim]
Let us assume without loss of generality that $K=I$, as the case $K=J$ is analogous. Note that if $K$ is 2-connected, then \eqref{eq:min_Deg_H[X]} and Lemma~\ref{lem:Dirac} together imply that $K$ has a matching covering at least $(1-200\delta)|I|\ge (1-400\delta)k/2$ vertices. If $K$ is not 2-connected, then there is some $i \in K$ such that $\red{R}[K\setminus\{i\}]$ has at least two components, say $K_1$ and $K_2$. Firstly, noting that \eqref{eq:min_Deg_H[X]} gives that $\red{\delta}[K_i] \ge (1-200\delta)|I|/2$ for $i\in[2]$, we conclude $\red{R}[K\setminus\{i\}]$ has exactly two components. Secondly, for $i\in [2]$ we have
\[|K_i|\le |I|-(1-200\delta)|I|/2\le (1+100\delta)|I|/2,\]
hence $K_i$ is non-bipartite and has a matching on at least $|K_i|-1\ge (1-20\delta)k/4$ vertices. 
\end{proof}
We now claim that there is at least one red edge between $I$ and $J$. Indeed, otherwise we have
\begin{align*}
    \delta(\blue{R}[I,J]) &\ge (3/4 - \delta)k - \max\{|I|, |J|\} - (k-|I\cup J|) \\
    & \ge (3/4 - \delta)k - (1+100\delta)k/2 - 5 \delta k\\ 
    &\ge (1/4 - 200\delta)k,
\end{align*} 
and so we are done by Lemma~\ref{cor:balanced_bipartite_all_trees}. 

Therefore, we may assume that there is at least one red edge $ij$ with $i\in I$ and $j\in J$. If at least one of $\red{R}[I]$ or $\red{R}[J]$ is 2-connected, then Claim~\ref{fact:structure} implies $\red{R}[I\cup J]$ contains a connected matching covering at least 
\[(1-400\delta)k/2+(1-20\delta)k/4\ge (1+100\delta)k/2\]
vertices. Say $K$ is this component. If $K$ is non-bipartite, then we are done by Lemma~\ref{lem:emb1}. Otherwise, $K$ is bipartite, in which case both $\red{R}[I]$ and $\red{R}[J]$ are $2$-connected and thus $V(K)=I\cup J$. Let $I'$ and $J'$ be the partition classes of $K$, and assume $|I'|\ge |J'|$. By \eqref{eq:min_Deg_H[X]}, we have  
\[|I'| \ge |J'| \ge (1-400\delta)k/2\quad\text{ and }\quad\delta(\red{R}[K])\ge (1-400\delta)k/2,\]
so we can apply Lemma~\ref{cor:balanced_bipartite_all_trees} to find a copy of $T$. 

We now suppose both $\red{R}[I]$ and $\red{R}[J]$ are not 2-connected, in which case we can partition $I = I_1 \cup I_2 \cup \{\Bar{i}\}$ and $J = J_1 \cup J_2 \cup \{\Bar{j}\}$ as in Claim~\ref{fact:structure}~\ref{item:X1-conn}.  If three of $I_1, I_2, J_1, J_2$ are connected in $\red{R}$, then these form a sufficiently large monochromatic non-bipartite red component so that we can conclude by Lemma~\ref{lem:emb1}.

Otherwise, each of $I_1, I_2, J_1, J_2$ is connected to at most one other in $\red{R}$. Recall we are assuming there is a red edge $ij$ with $i\in I$ and $j\in J$. Noting that $i\not=\Bar{i}$ and $j\not=\Bar{j}$, as otherwise we are in the former case, we must have a red edge between $I_1$ and  $J_1$, say. Then,  using \eqref{eq:mindeg:reduced} and $|I_1 \cup J_1|\le (1+200\delta)k/2$, we have that for $i' \in |I_1 \cup J_1|$, 
\begin{equation*}
    |N_R(i')\cap (I_2 \cup J_2)| \ge (3/4 - \delta)k - (1+200\delta)k/2\ge  (1-300\delta)k/4.
\end{equation*}
As we have assumed there are no red edges between $I_1 \cup J_1$ and $I_2 \cup J_2$, all of these edges must be blue. By the same argument, each vertex of $I_2 \cup J_2|$ has at least $(1-300\delta)k/4$ blue neighbours in $I_1 \cup J_1$. This gives $\delta(\blue{R}[I_1 \cup J_1, I_2 \cup J_2]) \ge (1-300\delta)k/4$. Hence, $\delta(\blue{R} -\{\Bar{i}, \Bar{j}\}) \ge (1-400\delta)k/4$, and, since $\blue{R}[I\cup J]$ is connected, we conclude by Lemma~\ref{cor:balanced_bipartite_all_trees}. \end{proof}

For trees with unbounded maximum degree, we cannot surpass case~\ref{structure3} as in the proof of Theorem~\ref{thm:stability}, and thus our stability result needs to deal with this situation.

\begin{theorem}\label{thm:stability:unbounded}Let $1/n\ll\alpha\ll\mu\ll \delta\le1$. Let $G$ be a blue/red coloured graph with $2n-2$ vertices and $\delta(G)\ge  (3/4+\delta)2n$ and suppose $T$ is an $n$-vertex tree with $\Delta(T)\le n^\alpha$. If $G$ contains no monochromatic copy of $T$, then, after possibly swapping colours, there is a partition $V(G)=U_0\cup U_1\cup U_2$ such that 
\begin{itemize}
    \item $|U_i|\ge (1-\mu)n$ for $i\in [2]$, 
  \item $\delta(\red{G}[U_i])\ge n/2+\delta n/100$ for $i\in [2]$, and
  \item $\red{e}(U_1,U_2)\le \mu n^2$.
\end{itemize}
\end{theorem}
\begin{proof}
 We introduce auxiliary constants $\eps,d>0$ so that $\alpha\ll \eps\ll d\ll \mu\ll \delta$. Let $G$ be a blue/red coloured graph on $2n-1$ vertices with $\delta(G)\ge (3/4+\delta)2n$ and let $T$ be an $n$-vertex tree with $\Delta(T)\le n^\alpha$. Use Theorem~\ref{thm:ColRegLem} with parameters $\varepsilon,d$ to get a partition $V(G)=V_0\cup V_1\cup\ldots\cup V_k$ so that \ref{item:(R1)}--\ref{item:(R5)} hold for some subgraph $G'\subset G$. Let $m=|V_1|=\ldots=|V_k|$ and note that~\ref{item:(R1)} gives $m\ge (1-\eps)|G|/k\ge (1-2\eps)2n/k$, and \ref{item:(R3)} gives $\delta(G')\ge (3/4+\delta)2n-(2d+\eps)|G|\ge (3/4+\delta/2)2n$, as $\eps,d\ll \delta$.

Let $R$ be the graph with vertex set $[k]$ where $ij\in E(R)$ if both $\blue{G'}[V_i,V_j]$  and $\red{G'}[V_i,V_j]$ are $\eps$-regular. For distinct $i,j\in [k]$, we colour  $ij$ blue if $\blue{G'}[V_i,V_j]$ is $(\eps,d)$-regular or red if $\red{G'}[V_i,V_j]$ is $(\eps,d)$-regular, breaking ties arbitrarily. Then, Lemma~\ref{fact:H_min_deg} gives
\begin{equation}\label{eq:mindeg:reduced2}\delta(R)\ge (3/4+\delta/2-(2d+2\eps))|R|\ge (3/4+\delta/4)|R|,
\end{equation}
as $\eps,d\ll\delta$, and thus Lemma~\ref{lem:H_strcuture} implies one of \ref{structure1}--\ref{structure4} holds. As in the proof of Theorem~\ref{thm:stability}, cases~\ref{structure1} and~\ref{structure3} gives a monochromatic copy of $T$, so let us focus here then on cases~\ref{structure2} and~case~\ref{structure4}.

We introduce an auxiliary constant $\delta'$ that satisfies $\eps,d\ll\delta'\ll\delta,\mu$. We quickly discard case~\ref{structure4}. Assume we are in case~\ref{structure4} so that there is a partition $V(R)=U_1\cup U_2\cup U_3\cup U_4$ such that (after possibly swapping colours) we have
\begin{itemize}
    \item $|U_i|\ge (1-\delta')k/4$ for $i\in [4]$,
    \item $\blue{e}(U_1\cup U_2, U_3\cup U_4)=0$, and
    \item $\red{e}(U_1\cup U_3, U_2\cup U_4)=0$.
\end{itemize}
Note that the first bullet point implies that $|U_i|\le k/4+\delta'k$ for $i\in [4]$. Then, \eqref{eq:mindeg:reduced2} gives 
\[e(U_1,U_4)\ge |U_1|\cdot \left((3/4+\delta/4)|R|-|U_1|-|U_2|-|U_3|\right)\ge |U_1|\cdot (\delta k/4-3\delta'k)\ge \delta k^2/100,\]
which contradicts either the second or third bullet point. 

Now suppose we are in case~\ref{structure2} so that there is a monochromatic bipartite component $C$ of $R$ such that $|C|\ge (1-\delta')k$ and $C$ contains a matching covering $(1/2+\delta')k$ clusters. Without loss of generality, we assume $C$ is a blue component, and let $I,J$ denote the bipartition classes of $C$ with $|I|\ge |J|$. 

Let $V_I=\cup_{i\in I}V_i$ and $V_J=\cup_{j\in J}V_j$, and note that $|V_I|\ge |V_J|$  as $|I|\ge |J|$. We claim $|V_I|\le n-1$. Indeed, note that
\begin{equation}\label{eq:mindeg:2}
\delta(\red{G}[V_I])\ge (3/4+\delta)|G|-\Delta(\blue{G}[V_I])-|V_J|-6\delta' n\ge (1/2+\delta/2)|V_I|,
\end{equation}
where we used that $\Delta(\blue{G}[V_I])\le (2d+\eps)|G|\le \delta n/100$ and $\delta'\ll \delta$. Hence, if $|V_I|\ge n$, Theorem~\ref{thm:KSS} implies $\red{G}[V_I]$ contains a copy of $T$. Observe that a similar calculation shows that \eqref{eq:mindeg:2} also holds for $\red{G}[V_J]$.

We may then assume $|V_J|\le |V_I|\le n-1$ from which we get $|V_I|\ge 2n-2-\eps |G|-3\delta'n-n\ge (1-4\delta')n$. Similarly, we get $|V_J|\ge (1-4\delta')n$ and thus the same calculation as in~\eqref{eq:mindeg:2} gives, for $K\in\{I,J\}$,
\begin{equation}\label{eq:mindeg:3}\delta(\red{G}[V_K])\ge (1/2+\delta/2)(1-4\delta')n\ge n/2+\delta n/100,\end{equation}
as $\delta'\ll \delta$. Set $U_1=V_I$ and $U_2=V_J$, and note that $|U_i|\ge (1-4\delta')n\ge (1-\mu)n$ as $\delta'\ll\mu$. Since~\eqref{eq:mindeg:3} holds for both $V_I$ and $V_J$, it is only left to show that $\red{e}(U_1,U_2)=\red{e}(V_I,V_J)\le \mu n^2$. To see this, we will show that there is no red edge between $I$ and $J$ in $R$. Indeed, if there is a red edge between $I$ and $J$, then $I\cup J$ induces a non-bipartite red component which contains a matching covering at least $(|I|-1)+(|J|-1)\ge (1-10\delta')k$ clusters, which is enough for applying Lemma~\ref{lem:emb1}.

\end{proof}
\section{Embedding trees in very dense graphs}\label{sec:extremal}
In this section, we prove several embedding results that we will need to analyse the extremal situations in Section~\ref{sec:final}. The first result finds a spanning tree in a graph which is almost a clique, and the second result is some sort of bipartite version of Theorem~\ref{thm:KSS}, which may be of independent interest. The last results are a bit different; these results find a large tree in a graph which is the union of an almost complete bipartite graph with either another almost complete bipartite graph or a relatively dense graph. 
The first two of these results use an \textit{absorption} approach inspired by the work of B\"otcher, Montgomery, Parzyck and Person~\cite{bottcher2020embedding} and of Kathapurkar and Montgomery~\cite{kathapurkar2022spanning}. The basic idea in the proof is to embed a small proportion of the tree randomly so that it gives enough flexibility to incorporate the last bit of the tree by locally \textit{switching} the role of a vertex which is currently embedded so that a new vertex can be added at the time. 
\begin{lemma}\label{lem:exact:1} Let $\Delta\ge 2$ and $1/n\ll \mu\ll \beta,1/\Delta$. Let $G$ be an $n$-vertex graph so that $\delta(G)\ge \beta n$ and for all but at most $\mu n$ vertices $x\in V(G)$, $d(x)\ge (1-\mu)n$. Then, $G$ contains a copy of every $n$-vertex tree $T$ with $\Delta(T)\le \Delta$. 
\end{lemma}
\begin{proof}Let us introduce first auxiliary constants $\gamma,\lambda$ such that $\mu\ll \lambda\ll\gamma\ll \beta,1/\Delta$. We use Lemma~\ref{lem:cutvertex:2} to find a vertex $s\in V(T)$ and subtrees $S,S'\subset T$ such that  $\gamma n\le |S|\le 2\gamma n$, $S\cup S'=T$ and $S$ and $S'$ intersect exactly at $s$. Label the vertices of $S$ as $s_1,\ldots, s_m$ so that $s_1=s$ and, for $i\in [m]$, $s_i$ has exactly one neighbour in $\{s_1,\ldots, s_{i-1}\}$. Moreover,  assume that for a set $J\subset \{4,\ldots, m\}$ of size $|J|=\lambda n$ we have that for each $j\in J$, $s_{j-1},s_j,s_{j+1}$ is a path in $T$, all neighbours of $s_j$ (except for $s_{j-1}$) appear right after $s_j$ in this ordering, and all vertices $\{s_j:j\in J\}$ are at distance at least $2$ in $T$. This is indeed possible as $\Delta(T)\le \Delta$ and $\lambda \ll \gamma,1/\Delta$. Our first step is to find an embedding $\psi:S\to G$ so that the following property holds.
\begin{enumerate}[label=\upshape{\textbf{P}}]
    \item\label{property:switching:1} For all distinct $x,y\in V(G)$, $|\{i\in [m]:\psi(s_i)\in N_G(x)\quad\text{and}\quad N_{\psi(S)}(\psi(s_i))\subset N_G(y)\}|\ge \lambda^2 n.$
\end{enumerate}
Let $X\subset V(G)$ be the set of vertices with at least $(1-\mu)n$ neighbours in $G$. Consider the following random process to find an embedding $\psi:S\to G$. First, select $v_1=\psi(s_1)$ uniformly at random from $X$. For $i\in [m]$, assume we have embedded $s_1,\ldots, s_{i-1}$ and, for $j\in [i-1]$, let $v_j=\psi(s_j)$ be the image of $s_j$ under this embedding. At step $i$, pick $v_i=\psi(s_i)$ uniformly at random from  $(N_G(v_j)\cap X)\setminus\{v_1,\ldots, v_{i-1}\}$, where $v_j=\psi(s_j)$ is the image of sole neighbour of $s_i$ in $\{s_1,\ldots, s_{i-1}\}$. Note that as $|S|\le 2\gamma n\le \delta (G[X])$, this process always succeeds in finding an embedding of $S$.
\begin{claim}$\psi:S\to G$ satisfies \ref{property:switching:1} with positive probability. \end{claim}

\begin{proof}Fix distinct vertices $x,y\in V(G)$. Let $J=\{j_1,\ldots, j_{\lambda n}\}$ and, for $i\in [\lambda n]$, let $X_i$ be the random variable that takes value 1 whenever $\psi(s_{j_i})\in N_G(x)$ and $N_{\psi(S)}(\psi(s_{j_i}))\subset N_G(y)$, and 0 otherwise. Let $i\in [\lambda n]$ and suppose that we have embedded $s_1,\ldots, s_{j_i-1}$. Then, the probability that $\psi(s_{j_i})\in N_G(x)$ and all the subsequent neighbours of $s_{j_i}$ belong to $N_G(y)$ is at least 
\[\frac{|(N_G(\psi({s_{j_i-1}}))\cap N_G(x)\cap X)\setminus \{\psi(s_1),\ldots,\psi(s_{{j_i}-1})\}|}{|N_G(\psi(s_{j_i-1}))\cap X|}\cdot \left(\frac{\beta n-\mu n-|S|}{n}\right)^{\Delta-1}\ge (\beta/2)^{\Delta},\]
where we used that $|V(G)\setminus X|\le \mu n$ and $\mu\ll\beta.$ Then, by conditioning on any possible outcome of $\psi(s_1),\ldots, \psi(s_{j_i-2})$ and that $s_{j_i-1}$ is embedded to $N_G(y)$, we have that
\[\mathbb P(X_i=1)\ge (\beta/2)^{\Delta+1}.\]
Set $\alpha=(\beta/2)^{\Delta+1}$, $Z_0=0$ and, for $j\in [\lambda n]$, let $Z_i=\sum_{j\in[i]}(X_j-\alpha)$. First, note that $(Z_j)_{j\ge 0}$ is a submartingale. Indeed, for all $i\ge 0$ we have 
\[\mathbb E[Z_{i+1}|Z_i,\ldots, Z_0]=Z_i+\mathbb E[(X_{i+1}-\alpha)|Z_i,\ldots, Z_0]\ge Z_i.\]
Noting that $|Z_{i}-Z_{i-1}|=|X_i-\alpha|\le 1$ for all $i\in [\lambda n]$, we have by Azuma's inequality (Lemma~\ref{lem:azuma}) for $t=\alpha \lambda n/2$ that 
\[\mathbb P\left(\textstyle\sum_{i\in[\lambda n]}(X_i-\alpha)\le t\right)\le 2\exp\left(\frac{-t^2}{2\lambda n}\right)=2\exp(-\alpha^2\lambda n/8).\]
Therefore, with probability at least $1-o(n^{-2})$ we have that $\sum_{i\in [\lambda n]}X_i\ge \alpha\lambda n+\alpha^2\lambda n/4\ge\lambda^2n$, as $\lambda\ll\beta,1/\Delta$. Taking a union bound over all choices of $x,y$ finishes the proof. 
   \end{proof} 
   Now, fix an embedding $\psi:S\to G$ so that \ref{property:switching:1} holds. Label $V(S')=\{s'_1,\ldots, s'_t\}$ so that $s'_1=s$ and, for each $i\in [t]$, $s'_i$ has a sole neighbour in $\{s'_1,\ldots, s'_{i-1}\}$. Using that $\delta(G[X])\ge n-2\mu n$, extend $\psi$ greedily so that the only missing vertices are $s'_{t-2\mu n+1},\ldots, s'_t$. Set $T_0=T-\{s'_{t-2\mu n+1},\ldots, s'_t\}$ and, for $i\in [2\mu n]$, let $T_i$ be tree formed by adding $s'_{t-2\mu n+i}$ to $T_{i-1}$. Letting $\psi_0=\psi$, we will find in turn an embedding $\psi_i:T_i\to  G$ for all $i\in [2\mu n]$. Label the leftover vertices of $G$ as $x_1,\ldots, x_{2\mu n}$ and, for each $i\in [2\mu n]$ in turn, do the following. Find an index $\ell_i\in [m]\setminus\{\ell_1,\ldots, \ell_{i-1}\}$ such that no vertex of $N_{\psi(S)}(\psi(s_{\ell_i}))$ has been changed so far and, if $s'_j$ is the sole neighbour of $s'_{t-2\mu n+i}$ in $\{s'_1,\ldots, s'_{t-2\mu n+i-1}\}$, 
   \[\psi(s_{\ell_i})\in N_G(\psi_{i-1}(s'_j))\quad\text{and}\quad N_{\psi(S)}(\psi(s_{\ell_i}))\subset N_G(x_i).\]
   Such an index always exists, as at each step we remove $\Delta+1$ possible candidates and $(\Delta+1)\cdot 2\mu n<\lambda^2n$, as $\lambda\gg \mu$. We then define $\psi_i(s'_{t-2\mu n+i})=\psi_{i-1}(s_{\ell_i})$, $\psi_i(s_{\ell_i}):=x_i$, and $\psi_i(x)=\psi_{i-1}(x)$ for every other vertex.
\end{proof}
\begin{lemma}\label{thm:KSS:bipartite} Let $1/n\ll \alpha\ll \mu\ll \delta\le 1$. Let $G$ be a bipartite graph with parts $U_1,U_2$ such that  $|U_1|=n-1$ and $n> |U_2|\ge (1-\mu)n$, and for every $i\in [2]$ and $u\in U_i$, $d(u)\ge (1+\delta)|U_{3-i}|/2$. If $T$ is an $n$-vertex tree with $\Delta(T)\le n^{\alpha}$, then $G$ contains a copy of $T$. 
\end{lemma}
\begin{proof}

Let $T$ be an $n$-vertex tree with $\Delta(T)\le n^{\alpha}$ and bipartition $V(T)=A\cup B$ such that $|A|\ge |B|$. Let $\mu'>0$ be an auxiliary constant with $\mu\ll\mu'\ll \delta$. If $|A|\le (1-\mu')n$, then $T$ is contained in $G$ by Lemma~\ref{lemma:bipartite:almost:spanning}. Assume otherwise that $|A|\ge (1-\mu')n$ and let $L$ be the number of leaves in $A$. Note that 
\[n-1=e(A,B)=\sum_{a\in A}d_T(a)\ge |L|+2|A\setminus L|=2|A|-|L|, \]
from which it follows that $|L|\ge 2|A|-n+1\ge n-3\mu' n$. We introduce  auxiliary constants $\gamma,\lambda$ with $\mu'\ll \lambda\ll\gamma\ll \delta$, and use Lemma~\ref{lem:cutvertex:2} to find a vertex $s_1\in V(T)$ and subtrees $S,S'\subset T$ such that  $\gamma n\le |S|\le 2\gamma n$, $S\cup S'=T$ and $S$ and $S'$ intersect exactly at $s_1$. Moreover, as at least $n-3\mu' n$ vertices of $A$ are leaves, $S$ contains a set $L_S\subset L$ of leaves not containing $s_1$ and such that $|L_S|=\gamma n/2$ . We label the vertices of $S$ as $s_1,\ldots, s_m$ so that for every $i\in [m]$, there is a unique vertex $s_j$ with $j<i$ such that $s_js_i\in E(T)$, and all vertices from $L_S$ appear at the end of this ordering. Let $\ell\in [m]$ such that $s_{\ell+1},\ldots,s_m$ are precisely the vertices in $L_S$. Our first step is to find an embedding $\psi:S\to G$ so that the following property holds.
\begin{enumerate}[label=\upshape{\textbf{Q}}]
    \item\label{property:switching} For all $y\in U_1$ and $x\in U_2$, $|\{\ell<i\le m:\psi(s_i)\in N_G(x)\quad\text{and}\quad N_{\psi(S)}(\psi(s_i))\subset N_G(y)\}|\ge \lambda n.$
\end{enumerate}
Consider the following random process to find an embedding $\psi:S\to G$. First, select $v_1=\psi(s_1)$ uniformly at random from $U_1$ if $s_1\in A$, and from $U_2$ otherwise. For $i\in [m]$, assume we have embedded $s_1,\ldots, s_{i-1}$, let $v_i=\psi(s_i)$ be the image of $s_i$ under this embedding. At step $i$, pick $v_i=\psi(s_i)$ uniformly at random from  $N_G(v_j)\setminus\{v_0,\ldots, v_{i-1}\}$, where $v_j=\psi(s_j)$ is the image of sole neighbour of $s_i$ in $\{s_1,\ldots, s_{i-1}\}$. Note that as $|S|\le 2\gamma n\le \delta (G)$, this process always succeeds in finding an embedding of $S$.

\begin{claim} $\psi:S\to G$ satisfies~\ref{property:switching} with positive probability.
\end{claim}

\begin{proof}Let us fix $y\in U_1$ and $x\in U_2$. For each $i\in [\ell]$, let $d_i$ be the number of neighbours of $s_i$ in $L_S$, and let $X_i$ be the random variable that takes value $d_i$ if $v_i\in N_G(y)$ and 0 otherwise. Note that if $d_i>0$, then $v_i$ is the image of some parent of leaves in $L_S$ and thus it is embedded in $U_2$. For each $i\in [\ell]$ such that $d_i>0$, given the choice of $v_1,\ldots, v_{i-1}$, we have that
\begin{equation}\label{eq:martingale}\mathbb P(X_i=d_i\mid X_{i-1},\ldots, X_1)=\frac{|N_G(y)\cap (N_G(v_j)\setminus \{v_1,\ldots, v_{i-1}\})|}{|N_G(v_j)\setminus \{v_1,\ldots, v_{i-1}\}|}\ge \frac{\delta n-2\gamma n}{n}\ge \delta/2,
\end{equation}
where $s_j$ is the sole neighbour of $s_i$ in $\{s_1,\ldots, s_{i-1}\}$. Therefore, we have $\mathbb E[X_i\mid X_{i-1},\ldots, X_{1}]\ge \delta d_i/2.$

Let $Z_0=0$ and, for $i\in [\ell]$, let $Z_i=\sum_{j\in [i]}(X_j-\delta d_j/2)$. Note that $\mathbb E[Z_{i+1}\mid Z_i,\ldots, Z_0]=Z_i+\mathbb E[(X_{i+1}-\delta d_{i+1}/2)\mid Z_i,\ldots, Z_0]\ge Z_i$
implies $(Z_i)_{i\ge 0}$ is a submartingale. Note that $\sum_{i\in [\ell]}d_i^2\le \Delta(T)\sum_{i\in [\ell]}d_i\le n^{\alpha}|L_S|=\gamma n^{1+\alpha}/2$ and $|Z_i-Z_{i-1}|=|X_i-\delta d_i/2|\le d_i$ for all $i\in [\ell]$. Therefore,  for $t=\delta \gamma n/8$, Azuma's inequality (Lemma~\ref{lem:azuma}) gives that 
\begin{equation}\label{eq:martingale:1}
    \mathbb P\left(\textstyle\sum_{i\in [\ell]}(X_i-\delta d_i/2)\le t\right)\le -2\exp\left(\frac{-t^2}{\sum_{i\in [\ell]}d_i^2}\right)\le 2\exp\left(\frac{-\delta ^2n^{1-\alpha}}{2\gamma}\right).
\end{equation}
Let $m'=\sum_{i\in [\ell]}X_i$ and condition on the event that $m'\ge \gamma^2 n $, which holds with probability $1-o(n^{-2})$ as $\gamma\ll \delta$.  Let $j_1,\ldots, j_{m'}\in\{\ell+1,\ldots,m\}$ be such that $y$ is adjacent to the parent of $v_{j_i}$ for $i\in [m']$. For each $i\in [m']$, let $Y_i$ be the random variable that takes value 1 if $v_{j_i}\in N_G(x)$ and 0 otherwise. A similar calculation as in~\eqref{eq:martingale} shows that $\mathbb P(Y_i=1\mid Y_{i-1},\ldots, Y_1)\ge \delta/2$. Therefore, letting $W_0=0$ and $W_i=\sum_{j\in [i]}(Y_i-\delta/2)$ for $i\in [m']$, we have 
\[\mathbb E[W_{i+1}\mid W_i,\ldots, W_0]=W_i+\mathbb E[(Y_{i+1}-\delta/2)\mid W_i,\ldots, W_0]\ge W_{i}.\]
Hence $(W_i)_{i\ge 0}$ is a submartingale and $|W_{i}-W_{i-1}|=|Y_i-\delta /2|\le \delta/2$  for $i\in [m']$. Azuma's inequality (Lemma~\ref{lem:azuma}), for $t=m'\delta /4$, gives
\begin{equation}\label{eq:martingale:2}\mathbb P\left(\textstyle\sum_{i\in [m']}(Y_i-\delta/2)\le t\right)\le 2\exp(-t^2/m')\le 2\exp(-\delta^2m'/16)\le 2\exp(-\delta^2\gamma^2n/16).\end{equation}
Then, letting $\ell'=\sum_{i\in [m']}Y_i$, we have that $\ell'\ge m'\delta/16\ge \gamma^3n$ with probability at least $1-2\exp(-\delta^2\gamma^2n/16)$, given that $m'\ge \gamma^2n$. Therefore,~\eqref{eq:martingale:1} and~\eqref{eq:martingale:2} imply
\[\mathbb P(\ell'\ge \gamma^3n\quad \text{and}\quad m'\ge \gamma^2n)\ge \mathbb P(\ell'\ge \gamma^3n|m'\ge \gamma^2n)\mathbb P(m'\ge \gamma^2n)\ge 1-o(n^{-2}).\]
Thus a union bound over all choices of $x$ and $y$ proves the claim.\end{proof}
Fix an embedding $\psi:S\to G$ that satisfies~\ref{property:switching} and note that $\psi(L_S)\subset U_1$. As $T$ contains a set $L\subset A$ of leaves with $|L|\ge n-3\mu' n$, we may remove a set $L'\subset V(S')\cap A$  of size $|L'|=\lambda^2 n$ that does not contain $s_1$. Set $T'=S'-L'$. Note that, for $i\in [2]$, every vertex in $U_i$ has degree at least 
\[(1+\delta)|U_{3-i}|/2-|S|\ge(1+\delta)|U_{3-i}|/2-2\gamma n\ge(1+\delta/2)|U_{3-i}|/2\]
outside $\psi(S)$, where we used that $|S|\le 2\gamma n\le \delta n/100$ and $|U_1|\ge |U_2|\ge (1-\mu)n$. Then, as $|V(T')\cap A|+|V(S)\cap A|\le n-\lambda^2 n$, we may use Lemma~\ref{lemma:bipartite:almost:spanning} to extend the embedding of $S$ into a embedding of $T-L'$.

Let $T_0=T-L'$ and let $\psi_0:T_0\to G$ be the embedding we have so far. Label $L'$ as $\ell_1,\ldots, \ell_t$, where $t=\lambda^2n$, and let $y_1,\ldots, y_t$ be some arbitrary vertices in $U_1$ not used by $\psi_0(T_0)$. For $i\in [t]$, we will find an embedding $\psi_i:T_i\to G$ where $T_i$ is the tree obtained by adding $\ell_i$ to $T_{i-1}$ and $V(\psi_i(T_i))=V(\psi_{i-1}(T_{i-1}))\cup\{y_i\}$. For $i\in [t]$ in turn, let $p_i\in U_2$ be the image in $\psi_{i-1}$ of the parent of $\ell_i$ and find some $j_i\in\{\ell+1,\ldots,m\}\setminus \{j_1,\ldots, j_{i-1}\}$ such that 
\[v_{j_i}\in N_G(p_i)\quad\text{and}\quad N_{\psi_{i-1}(T_{i-1})}(v_{j_i})\subset N_G(y_i).\]
Then define $\psi_i$ by setting $\psi_i(\ell_i)=v_{j_i}$  and letting $y_i$ take the role of $v_{j_i}$ in $\psi_{i}$, which completes the $i$-th step. This will be possible as property~\ref{property:switching} gives at least $\lambda n>t$ choices for $j_i$ at step $i\in [t]$. Note that after $t$ steps, $\psi_t(T_t)$ is indeed a copy of $T$, which thus finishes the proof.
\end{proof}
Now we prove the last two results of this section. In these results, we embed a tree into a graph which consists of two very dense bipartite graphs which some additional edges. The proofs use randomised semi-greedy arguments, where we first partition the tree into small subtrees and then assign those subtrees randomly to one of the bipartite graphs where most of the subtree is embedded. We then show that with positive probability such an embedding is possible.
\begin{lemma}\label{lem:emb:extremal:0}Let $1/n\ll \alpha\ll \mu\ll \beta\ll\gamma\le 1$ and let $T$ be a tree with $(1-\gamma)n$ vertices and $\Delta(T)\le n^\alpha$. Suppose $G$ is a graph with at most $n$ vertices that contains pairwise disjoint sets $U_1,U_2$ such that 
\begin{enumerate}[label=\upshape{(\roman{enumi})}]
    \item $|U_i|\ge (1-\mu)n/2$ for $i\in [2]$,
    \item $\delta(G[U_1,U_2])\ge (1-\mu)n/2$, and 
    \item $e(U_1)\ge 100\beta n^2$.
\end{enumerate}
Then, for every $t\in V(T)$ and $u\in V(G)$, there is an embedding of $T$ in $G$ with $t$ copied to $u$.
\end{lemma}
\begin{proof}
     First, we observe that $U_1$ contains a subset $W$ of size $|W|\ge \beta^2 n$ such that each vertex $w\in W$ satisfies $d(u,U_1)\ge \beta^2n$, as otherwise we have 
\[200\beta n^2\le 2e(U_1)\le |U_1|\cdot \beta^2n+\beta^2 n\cdot \Delta(G[U_1])\le 8\beta^2n^2,\]
a contradiction. Let $X_1,X_2\subset U_1$ be two disjoint random subsets where an element of $U_1$ either belongs to $X_1$ with probability $p$, or belongs to $X_2$ with probability $p$, or to $U_1\setminus (X_1\cup X_2)$ otherwise, and all choices are made independently. By Lemma~\ref{lemma:chernoff}, there is a choice of $X_1,X_2$ for which 
\begin{itemize}
    \item $pn/2\le |X_i|\le 2pn$ and $|X_i\cap W|\ge p\beta^2n/10$ for $i\in [2]$,
\item every  $w\in W$ satisfies $d(w,X_i)\ge \beta^2pn/2$, and
    \item every  $x\in U_2$ satisfies $d(x,U_1\setminus (X_1\cup X_2)\ge (1-10\mu)n/2$ and $d(x,X_i)\ge (p-\mu)d(x,U_1)$ for $i\in [2]$.
\end{itemize}
Noting that the first and second bullet points together imply 
\[e(X_1\cap W,X_2\cap W)\ge \frac{\beta^2pn}{2}\cdot|W\cap X_1|\ge \beta^4p^2n^2/4,\]
hence there are sets $W_1\subset W\cap X_1$ and $W_2\subset W\cap X_1$ such that $\delta(G[W_1,W_2])\ge \Omega(\beta^4pn)$.

Let us now turn our attention to $T$. We define $T_0=\{t\}\cup N_{T}(t)\cup N_{T}^2(t)$ and note that $|T_0|\le 1+\Delta(T)+\Delta(T)^2\le n^{1/100}$ provided $\alpha\ll 1$. Let $F=T-T_0$ and root each component of $F$ at vertices of distance 3 from $t$ in $T$. For an auxiliary constant $\nu>0$ with $\alpha\ll\nu\ll\mu$, we use Lemma~\ref{beta_decomp_v2} to find a $\nu$-decomposition $F=T_1\cup \ldots \cup T_s$ such that $s\le 100\nu^{-2}$ and $|T_i|\le \nu n$ for all $i\in [s]$. We now produce an embedding $\psi:T\to G$ so that $\psi(t)=u$. Let $V(T)=A\cup B$ be the unique bipartition of $T$. Without loss of generality, we assume $u\in U_2$, so we set $\psi(t)=u$ and then embed $\psi(T)$ so that $N_{T}(t)$ is mapped to $U_1$ and $N^2_{T}(t)$ is mapped to $U_2$.  
Assume $T_1,\ldots, T_s$ are ordered so that $T_0\cup T_1\cup \ldots\cup T_i$ is connected for each $i\in [s]$. Let $S$ be the connecting vertices of $T_1,\ldots, T_s$, that is, the set of those vertices appearing in at least two distinct $T_i$'s.  

For $i\in[s]$, let $\eps_i\in\{0,1\}$ be chosen uniformly at random, making all choices independently. For each $i\in [s]$ in turn, embed $T_i$ into $G$ under the following rules (if possible).\stepcounter{propcounter}
\begin{enumerate}[label=\upshape{\textbf{\Alph{propcounter}\arabic{enumi}}}]
    \item\label{G:ext:1} Every vertex in $S\cap V(T_i)$ is mapped to $W_2$. 
    \item\label{G:ext:2} If $\eps_i=0$, then each vertex $x\in V(T_i)\setminus (S\cup N_{T}(S))$ is mapped to $U_1\setminus (X_1\cup X_2)$ if $x\in A$ and mapped to $U_2$ if $x\in B$. If $\eps_i=1$, then each vertex $x\in V(T_i)\setminus (S\cup N_{T}(S))$ is mapped to $U_2$ if $x\in A$ and mapped to $U_1\setminus (X_1\cup X_2)$ if $x\in B$.
    \item\label{G:ext:3} If $x\in V(T_i)\cap A$ is the neighbour of some vertex in $S$, then $x$ is mapped to $W_1$ whenever $\eps_i=0$, and mapped to $U_2$ if $\eps_i=1$. If $x\in V(T_i)\cap B$ is the neighbour of some vertex in $S$, then $x$ is mapped to $U_2$ whenever $\eps_i=0$, and mapped to $W_1$ if $\eps_i=1$.     
    \end{enumerate}
We claim that with probability at least $1/2$,~\ref{G:ext:1}--\ref{G:ext:3} extend the embedding of $T_0$ to an embedding $\psi:T\to G$. For each $i\in [s]$, define the random variable 
\[Z_i=|V(T_i)\cap A|\cdot(1-\eps_i)+|V(T_i)\cap B|\cdot \eps_i.\]
Noting that $\mathbb E[Z_i]=\frac{1}{2}|T_i|$ for each $i\in [s]$, the random variable $Z=Z(\eps_1,\ldots, \eps_s)=\sum_{i\in [n]}Z_i$ satisfies $\mathbb E[Z]=|F|/2\le (1-\gamma)n/2$. As changing the value of $(\eps_1,\ldots, \eps_s)$ in one entry changes the value of $Z$ by at most $|T_i|\le \nu n$ and noting that $\sum_{i\in [s]}|T_i|^2\le \nu n\sum_{i\in [s]}|T_i|\le 2\nu n^2$, McDiarmid's inequality (Lemma~\ref{lemma:mcdiarmid}) gives  
\[\mathbb P(Z>n/2-\gamma n/100)\le \mathbb P(Z>\mathbb E[Z]+\gamma n/2)\le \exp(-\gamma^2n^2/(16\nu n^2))\le 1/2\]
provided $0<\nu\ll\gamma\le 1$. 

Conditioned on the event $Z\le n/2-\gamma n/100$, we show that $\psi$ can be successfully extended. Suppose that $\psi:T_0\cup T_1\cup\ldots \cup T_i\to G$ is defined for some $1\le i<s$, and try to embed $T_{i+1}$ following~\ref{G:ext:1}--\ref{G:ext:3}. Let $r_i\in S$ be the root of $T_{i+1}$, which is already embedded into some vertex $\psi(r_i)\in W_2$ by~\ref{G:ext:1}. For the sake of the argument, let us assume $\eps_{i+1}=0$. We embed $T_{i+1}$ greedily following~\ref{G:ext:2} and~\ref{G:ext:3}. When embedding a vertex $v\in (V(T_{i+1})\cap A)\setminus (S\cup N_T(S))$, as its parent is embedded into $U_2$, the number of choices for $\psi(v)$ is at least 
\[(1-10\mu)n/2-\sum_{j\in [s]:\eps_j=0}|V(T_j)\cap A|-\sum_{j\in [s]:\eps_j=1}|V(T_j)\cap B|-|X_1\cup X_2|\ge \mu n/100,\]
as $Z\le n/2-\gamma n/100$, $\mu\ll\gamma$ and $|X_1\cup X_2|\le 4pn\le \mu n/100$, and thus we can safely embed $v$. When the parent of $v$ is embedded into $U_1$, the argument is completely analogous. 

Now suppose we are to embed a vertex $v\in (S\cup N_T(S))\cap V(T_{i+1})$ and let $w\in V(T_{i+1})$ be the parent of $v$. Note that $d(\psi(w),W_i)\ge (p-\mu)(1-\mu)n/2\ge \beta^2pn/100$ for $i\in [2]$, and the number of occupied vertices in $W_1\cup W_2$ is at most 
\[|S|+|N_T(S)|\le \nu^{-1}+\nu^{-1}\Delta(S)\le\beta^2pn/200,\]
thus we can safely choose an unused vertex $\psi(v)\in N(\psi(w))\cap W_1$, if $v\in S$, or $\psi(v)\in N(\psi(w))\cap W_2$, if $v\in N_T(S)$.

\end{proof}

\begin{lemma}\label{lem:emb:extremal:3}Let  $1/n \ll \mu\ll 1/C\ll 1/\Delta$ and let $T$ be an $n$-vertex tree with $\Delta(T)\le \Delta$. Suppose $G$ contains pairwise disjoint sets $U_1,U_2, V_1, V_2$ which themselves contain subsets $U_i' \subset U_i$ and $V_i' \subset V_i$ such that 
\begin{enumerate}[label=\upshape{(\roman{enumi})}]
    \item $|U_1\cup V_1|=n-1$ and $|U_2\cup V_2|=n-1$,
    \item $|U_i'| \le \mu n/2$ and $|V_i'| \le \mu n/2$ for $i \in [2]$,
    \item $\delta(G[U_1 \setminus U_1' ,U_2 \setminus U_2' ])\ge (1-2\mu)n/2$ and $\delta(G[V_1\setminus V_1',V_2\setminus V_2'])\ge (1-2\mu)n/2$,
    \item $\delta(G[U_1 ,U_2 ])\ge (1-5\mu)n/4$ and $\delta(G[V_1,V_2])\ge (1-5\mu)n/4$, and
    % \item $\delta(G[U_1 ,U_2 ])\ge (1-\mu)n/2$ and $\delta(G[V_1,V_2])\ge (1-\mu)n/2$, and
    \item $d(G[U_1, V_2]) \ge C $.
\end{enumerate}
Then, $G$ contains a copy of $T$.
\end{lemma}
\begin{proof}

%Fix some $\beta > 2\Delta ^2/n$, and let $p$ be such that $1/\Delta \gg p \gg 1/\beta $. 
Throughout the proof we will assume that $\Delta\ge 3$, as when $\Delta=2$ the tree $T$ is an $n$-vertex path, in which case we can embed $T$ greedily. 

\begin{claim}There are sets $X_1\subset U_1$ and $Y_2\subset V_2$ such that the following properties hold.
\begin{enumerate}[label=\upshape{(\alph{enumi})}]
    \item $|X_1|,|Y_2|\le \frac{n}{100\Delta}.$
    \item $\delta(G[X_1,Y_2])\ge \frac{C}{400\Delta}$. 
    \item For all but at most $n/4$ vertices $x\in U_2$, $d(x,X_1)\ge C/(20000\Delta)$.
    \item For all but at most $n/4$ vertices $x\in V_1$, $d(x,Y_2)\ge C/(20000\Delta)$.
\end{enumerate}
    
\end{claim}
\begin{proof}[Proof of Claim]
Let $H\subset G[U_1,V_2]$ be a subgraph of minimum degree at least $C/2$ and let $A=V(H)\cap U_1$ and $B=V(H)\cap V_2$. Form subsets $B_p \subset B$ and $A_p\subset A$ by picking vertices independently at random with probability $p=1/200\Delta$. Fix $x\in A$. Then,
$\EE[d(x, B_p)] \ge Cp/2$ and hence, by Lemma~\ref{lemma:chernoff}, we have $\PP(d(x, B_p) \le Cp/4) \le 2e^{-p C/24}$. Let $Z_A  = \{x \in A_p:d(x, B_p) \le Cp/4\}$ and note that $\EE[|Z_A|] \le 2|A|e^{-pC/24}.$ Then, by Markov's inequality, we have
\[\mathbb P\left(|Z_A|\ge 100|A|e^{-pC/24}\right)\le 1/100,\]
Similar calculation shows that for $Z_B  = \{x\in B:d(x, A_p) \le Cp/4\}$, we have $|Z_B|\le 100|B|e^{-pC/24}$ with probability more than $99/100$. Therefore, with probability more than $98/100$ we have that $|Z_A|\le1 00|A|e^{-pC/24}$ and $|Z_B|\le 100|B|e^{-pC/24}$. Moreover, Lemma~\ref{lemma:chernoff} implies that $|A_p|=(1\pm0.1)p|A|$ with probability at least $1-\exp(-\Theta(p|A|))> 0.99$, as $p|A|\ge pC/2$ and $C\gg \Delta$. Similarly, with probability more than $0.99$ we have $|B_p|=(1\pm 0.1)p|B|$. We therefore have with probability at least 0.95
\begin{itemize}
\item$|A_p|=(1\pm0.1)p|A|$ and $|Z_A|\le 100|A|e^{-pC/24}$, and 
\item $|B_p|=(1\pm0.1)p|B|$ and $|Z_B|\le 100|B|e^{-pC/24}.$
\end{itemize}
First, we note that for any such choice of $A_p$ and $B_p$, we have $d(G[A_p,B_p])\ge Cp/8$. Indeed, we have that 
\[2e(A_p,B_p)\ge \sum_{x\in A_p\setminus Z_A}d(x,B_p)+\sum_{y\in B_p\setminus Z_B}d(x,A_p)\ge \frac{Cp}{8}(|A_p|+|B_p|),\]
as $|Z_A|\le 100|A|e^{-pC/24}\le 0.01p|A|\le 0.01|A_p|$ provided $C\gg \Delta$ and similarly for $Z_B$. Let $H'\subset G[A_p,B_p]$ be a subgraph with $\delta(H')\ge Cp/16$ and let $A'=V(H')\cap U_1$ and $B'=V(H')\cap V_2$.
Let $W_U=\{x\in U_2:d(x,A')\le Cp/1000\}$ and $W_V=\{x\in V_1:d(x,B')\le Cp/1000\}$. By double-counting, we have 
\[(|U_2|/2-2\mu n)|A'|\le e(A,U_2)\le |W_U|\cdot\frac{Cp}{1000}+(|U_2|-|W_U|)|A'|,\]
which gives $|W_U|\le \frac{|A'|}{|A'|-Cp/1000}(|U_2|/2 + 2\mu n)\le 6|U_2|/10$, where we used that $|A'|\ge \delta(H')\ge Cp/16$. The same argument shows that $|W_V|\le 6|V_1|/10$, which finishes the proof by setting $X_1=A'$ and $Y_2=B'$.

%We also have that for each $x \in U_2 \setminus W_U$, $\EE[d(x,A_p)] \ge pC/10$, and thus by Lemma~\ref{lemma:chernoff} we have $\PP(d(x, A_p) < pC/20) < 2 e^{-pC/24}$. Letting $S_A = \{x \in U_2 : d(x, A_p) \le C/(4000\Delta)\}$, we have $\EE[|S_A|] < 2|U_2| e^{-pC/24}$. Again, we can use Markov's inequality and the fact that $pC \gg 1$ to obtain that 
%\[\PP(|\{x \in U_2 : d(x, A_p) \le C/(4000\Delta)\}| \ge n/4) < 1/100.\]
%Hence, by~\eqref{Ap:prop1}, with probability at least 0.9 all but at most $n/4$ vertices of $U_1$ have degree at least $C/5000$ into $A_p \setminus Z_A$.

% A similar argument shows that with probability at least $0.9$ all but at most $n/4$ vertices of $V_1$ satisfy $\EE[d(x,B_p\setminus Z_B)] \ge C/5000$. Recalling~\eqref{Ap:prop1} and~\eqref{Bp:prop1}, there is some choice of $A_p$ and $B_p$ which give the claim by setting $X_1 = A_p \setminus Z_A$ and $Y_2 = B_p \setminus Z_B$.

\end{proof}

\begin{claim}There are subsets $X_2,X_1',X_2'$ and $Y_1,Y_1',Y_2'$, all pairwise disjoint and disjoint from $X_1$ and $Y_2$, such that $X_2 \subset \{x \in U_2\setminus U_2': d(x, X_1) > C/20000\Delta\}$ and $Y_1 \subset  \{x \in V_1\setminus V_1': d(x, Y_2) > C/20000\Delta\}$, $X'_i\subset U_i\setminus U_i'$, $Y_i'\subset V_i\setminus V_i'$ for $i\in [2]$, and the following properties hold. 
\begin{enumerate}[label=\upshape(\alph{enumi})]
    \item $10\mu n\le |X_2|,|X_1'|,|X'_2|\le 100\mu n$ and $10\mu n\le|Y_1|,|Y'_1|,|Y_2'|\le 100\mu n$.
      \item For $i\in [2]$ and for all $x\in U_i$, $d(x,X_{3-i}')\ge 4\mu n$.
      \item For $i\in [2]$ and for all $x\in V_i$, $d(x,Y_{3-i}')\ge 4\mu n$.
   
\end{enumerate}
\end{claim}
\begin{proof}Form $X_1'$ by including each vertex $u\in U_1\setminus (X_1\cup U_1')$ independently with probability $50\mu $, and form $X_2'\subset U_2$ by including each vertex $u\in U_2$ independently with probability $50\mu $.  By Lemma~\ref{lemma:chernoff}, with probability more than 0.9 we have that $10\mu n\le |X_1'|,|X'_2|\le 100\mu n$. Fix $i\in [2]$ and $x\in U_i$. Noting that $\mathbb E[d(x,X_{3-i}')]\ge (50\mu) \cdot n/10\ge 5\mu n$, Lemma~\ref{lemma:chernoff} implies that $d(x,X_{3-i}')\ge 4\mu n$ with probability $1-e^{-\Theta(n)}$. Thus, by an union bound we have that with probability at least 0.9, $d(x,X_{3-i}')\ge 4\mu n $ for all $x\in U_i$. Finally, take any set $X_2\subset \{x \in U_2\setminus U_2': d(x, X_1) > C/20000\Delta\}\setminus X_2'$ of size $|X_2|=10\mu n$, which exists as at most $n/4$ vertices in $U_2$ have low degree to $X_1$ and $|U_2'|\le \mu n/2$. The sets $Y_1,Y_1',Y_2'$ can be found analogously.
    
\end{proof}

We  now introduce an auxiliary constant $\beta>0$ such that $ \mu\ll1/C\ll\beta\ll 1/\Delta$. Let $T_1,\ldots, T_s$ be a vertex-disjoint collection of subtrees of $T$ with roots $R =\{r_1, \ldots, r_s\}$ as given by Lemma~\ref{vtx_disjoint_beta_decomp}, such that $R$ is $10$-independent in $T$ and $s \le 2\Delta\beta^{-1}$. Furthermore, suppose that $T_1,\ldots, T_s$ are ordered so that for each $i\ge2$, the parent of $r_i$ is contained in $T_{i-1}$. Let $P = \{p_2, \ldots, p_{s}\}$ be the set of parents of each $r_i$ in $T$, and $Q$ be the set of parents of the elements of $P$. We now describe how we will assign these pieces to $G$.

Let $V(T)=A_1\cup A_2$ be the bipartition classes of $T$ and assume without loss of generality that $|A_1|\ge\lfloor n/2\rfloor\ge |A_2|$. For each $i \in [s]$, we define a random variable $\eps_i \in \{0, 1\}$ chosen uniformly at random, and independently of each other. Now, for each $i \in [s]$ in turn, we define an embedding $\psi_i : T_1 \cup \ldots \cup T_i \to G$ (if possible), satisfying the following rules. We state the rules in the case that $r_i\in A_j$ for $j=1$, the other case is analogous.
\stepcounter{propcounter}
\begin{enumerate}
[label=\upshape{\textbf{\Alph{propcounter}\arabic{enumi}}}]
\item\label{rule:sides} Every vertex in $A_1$ is mapped to $U_1\cup V_1$ and every vertex in $A_2$ is mapped to $U_2\cup V_2$. Moreover, every vertex in $P$ is mapped to $(X_1'\cup X_2')\cup (Y_1'\cup Y_2')$. 
    \item\label{rule:eps=0:1} If $\eps_i=0$ and $\psi_i(p_i)\in U_1\cup U_2$, then $r_i$ is embedded to $N(\psi_{i-1}(p_{i}))~\cap~X'_1$ and then every vertex in $T_i-P$ is greedily embedded in $G[U_1\setminus (U_1'\cup X_1\cup X_1'),U_2\setminus (U'_2\cup X_2\cup X_2')]$  and vertices in $V(T_i)\cap P$ are mapped to either $X_1'$ or $X_2'$, accordingly.  
    \item\label{rule:eps=0:2} If $\eps_i=0$ and $\psi_i(p_i)\in V_1\cup V_2$, then $r_i$ is embedded to $N(\psi_{i-1}(p_{i}))~\cap~Y_1$. We embed $N_{T_i}(r_i)$ to $Y_2$, $N^2_{T_i}(r_i)$ to $X_1$, and $N^3_{T_i}(r_i)$ to $X_2'$. Then every other vertex of $T_i-P$ is  greedily embedded in $G[U_1\setminus (U_1'\cup X_1\cup X_1'),U_2\setminus (U_2'\cup X_2\cup X_2')]$ and vertices in $V(T_i)\cap P$ are mapped to either $X_1'$ or $X_2'$, accordingly.  
     \item\label{rule:eps=1:1} If $\eps_i=1$ and $\psi_i(p_i)\in V_1\cup V_2$, then $r_i$ is embedded to $N(\psi_{i-1}(p_{i}))~\cap~Y'_1$ and then every vertex in $T_i-P$ is greedily embedded in $G[V_1\setminus (V_1'\cup Y_1\cup Y_1'),V_2\setminus (V_2'\cup Y_2\cup Y_2')]$ and vertices in $V(T_i)\cap P$ are mapped to either $Y_1'$ or $Y_2'$ accordingly.  
    \item\label{rule:eps=1:2} If $\eps_i=1$ and $\psi_i(p_i)\in U_1\cup U_2$, then $r_i$ is embedded to $N(\psi_{i-1}(p_{i}))~\cap~X'_1$. We embed $N_{T_i}(r_i)$ to $X_2$, $N^2_{T_i}(r_i)$ to $X_1$, $N^3_{T_i}(r_i)$ to $Y_2$ and $N^4_{T_i}(r_i)$ to $Y_1'$. Then every other vertex of $T_i-P$ is  greedily embedded in $G[V_1\setminus (V_1'\cup Y_1\cup Y_1'),V_2\setminus (V_2'\cup Y_2\cup Y_2')]$ and vertices in $V(T_i)\cap P$ are mapped to either $Y_1'$ or $Y_2'$, accordingly.  
\end{enumerate}
For each $i \in [s]$, define the following random variables $M_i = |V(T_i) \cap A_1| \cdot (1-\eps_i)$, $W_i = |V(T_i) \cap A_2| \cdot (1- \eps_i)$, $M'_i = |V(T_i) \cap A_1| \cdot \eps_i$, and  $W'_i = |V(T_i) \cap A_2| \cdot \eps_i$. Let $M = \sum_{i \in [s]}M_i$, $W = \sum_{i \in [s]}W_i$, $M' = \sum_{i \in [s]}M'_i$, and $W' = \sum_{i \in [s]}W'_i$. Note that $\EE[M_i] = \frac{1}{2}|V(T_i) \cap A_1|$ and $\EE[M] = \frac{1}{2}|A_1| \le \frac{\Delta-1}{\Delta}\frac{n}{2} + \frac{1}{2 \Delta} \le (1-\frac{1}{2\Delta})\frac{n}{2}$. Note that changing the value of one $\eps_i$ changes the value of $M=M(\eps_1, \ldots, \eps_s)$ by at most $|V(T_i) \cap A_1| \le \beta n$. Also, $\sum_{i \in [s]}|T_i|^2 \le \beta n\sum_{i \in [s]}|T_i| \le \beta n^2$. Thus, by McDiarmid's inequality, 
\[
\PP\left(M > \left(1 - \frac{1}{4\Delta}\right)\frac{n}{2}\right) \le 2\exp\left(-\frac{n^2}{64\Delta^2\beta n^2}\right)< 1/4,
\]
as $\beta \ll 1/\Delta$. Likewise,  we have $\PP\left(M' > \left(1 - \frac{1}{4\Delta}\right)\frac{n}{2}\right)< 1/4$ and $\PP~\big(W > 0.26n  \big),\PP~\big(W' > 0.26n\big)~<~1/4$. 

Therefore, we may fix a choice of $(\eps_1, \ldots, \eps_s)$ for which $M, M' \le  \frac{n}{2} - \frac n {8\Delta} $ and $W,W'\le 0.26n$. We now show that the embedding $\psi_i$ can be successfully extended for each $i$ for this choice of $(\eps_1, \ldots, \eps_s)$. First, note $T_1$ can easily be embedded satisfying \ref{rule:sides}--\ref{rule:eps=1:2}. 

Now let $i \in [s]$ and consider $\psi_{i+1}$. Assume $\eps_{i+1} = 1$ and $r_i\in A_1$, as the other case is the symmetric. Assume that $\psi_i(p_i)\in U_1\cup U_2$, in which case, as $r_i\in A_1$, \ref{rule:sides} implies $\psi_i(p_i)\in X_2'$. The key observation here is that for every $i\in [s]$ we have that
\begin{equation}\label{embedding:space:constraint}|\psi_i^{-1}(X_1\cup X'_1\cup X_2\cup X_2'\cup Y_1\cup Y_1'\cup Y_2\cup Y_2')|\le s\cdot (1+\Delta+\Delta^2+\Delta^3+\Delta^4)\le 10\Delta^6\beta^{-1}\le \frac{C}{10^{10}\Delta}, \end{equation}
as $1/C\ll \beta,1/\Delta$.
Then, as $\psi_i(p_i)\in X_2'$ has at least $4\mu n$ neighbours in $X_1'$,~\eqref{embedding:space:constraint} implies we can safely embed $r_{i+1}$ to some $\psi_{i+1}(r_{i+1})\in X_1'$. Now, as $X_1'\subset U_1\setminus U_1'$, $\psi_{i+1}(r_{i+1})$ has at most $3\mu n$ non-neighbours in $U_2$ and thus $d(\psi_{i+1}(r_{i+1}),X_2)\ge |X_2|-3\mu n\ge 7\mu n$. Because of~\eqref{embedding:space:constraint}, we then can greedily embed every vertex in $N_{T_{i+1}}(r_{i+1})$ into $X_2$. The condition on $X_2$ implies every vertex in $\psi_{i+1}(N_{T_{i+1}}(r_{i+1}))$ has at least $C/(20000\Delta)$ neighbours in $X_1$, thus~\eqref{embedding:space:constraint} implies we have enough room to embed $N^2_{T_{i+1}}(r_{i+1})$ into $X_1$. Again, the minimum degree between $X_1$ and $Y_1$ is at least $C/400\Delta$, which implies we can embed $N^3_{T_{i+1}}(r_{i+1})$ to $Y_2$, and then as every vertex in $Y_2$ has at least $4\mu n$ neighbours in $Y_1'$, we can extend the embedding of the first 4 levels of $T_{i+1}$ under the rule~\ref{rule:eps=1:2}. Now remains to embed the rest of $T_{i+1}-P$ greedily in $G[V_1\setminus (V_1'\cup Y_1\cup Y_1'),V_2\setminus (V_2'\cup Y_2\cup Y_2')]$, which is possible as the minimum degree in that graph is at least 
\[(1-2\mu n)\frac{n}{2}-200\mu n-\mu n/2\ge \left(1-\frac{1}{4\Delta}\right)\frac{n}{2}\ge\max\{M',W'\},\]
by the choice of $(\eps_1,\ldots,\eps_s)$ and using that $\mu\ll 1/\Delta$. Finally, as every vertex in $V_1$ (resp. $V_2$) has at least $4\mu n$ neighbours in $Y_2'$ (resp. $Y_1'$), we can embed the vertices in $P\cap V(T_{i+1})$ under rule~\ref{rule:eps=1:2}. The other cases are completely analogous.

As~\eqref{embedding:space:constraint} holds for every $i\in [s]$, we can find an embedding $\psi_s:T\to G$ under rules~\ref{rule:sides}--\ref{rule:eps=1:2}, finishing the proof.
\end{proof}

\section{Proof of the main theorems}\label{sec:final}
In this section, we put everything together to prove Theorems~\ref{thm:main},~\ref{thm:main3} and~\ref{thm:main2}. We first need two results that will allow us to clean up the structure given by Theorems~\ref{thm:stability} and~\ref{thm:stability:unbounded}. After cleaning up the partitions given by~Theorems~\ref{thm:stability} and~\ref{thm:stability:unbounded}, we can directly use the results from Section~\ref{sec:extremal} to embed a tree in one of the colours.
% \begin{proposition}\label{lem:emb:extremal:1}Let $1/n\ll \alpha\ll \mu\ll \beta\le 1$ and let $T$ be an $n$-vertex tree with $\Delta(T)\le n^\alpha$. Suppose $G$ is a graph with at most $4n$ vertices that contains pairwise disjoint sets $U_1,U_2,V_1,V_2$ such that 
% \begin{enumerate}[label=\upshape{(\roman{enumi})}]
%     \item $|U_i|,|V_i|\ge (1-\mu)n/2$ for $i\in [2]$,
%     \item $\delta(G[U_i,V_i])\ge (1-\mu)\min\{|U_i|,|V_i|\}$, and 
%     \item $e(U_1)\ge 100\beta n^2$.
% \end{enumerate}
% If there is some $x_0\in V(G)$ such that $d(x_0,U_i\cup V_i)\ge 1000\Delta(T)$ for $i\in [2]$, then $G$ contains a copy of $T$. 
% \end{proposition}
% \begin{proof}
% Apply Lemma~\ref{lem:cutvertex:2} with $\gamma=\frac{1}{100}$ to get a vertex $z\in V(T)$ and a subtree $T'\subset T$ such that 
% \begin{itemize}
%     \item $n/100\le |T'|\le n/50$, and
%     \item $T$ and $T'$ intersect exactly at $z$. 
% \end{itemize}
% First, embed $T'$ into $U_2\cup V_2$ with $z$ copied to $x_0$. Then, use Lemma~\ref{lem:emb:extremal:0} to embed $T-(T'-z)$ into $U_1\cup V_1$ with $z$ copied to $x_0$. 
% \end{proof}

\begin{proposition}\label{lem:emb:extremal:1}Let $1/n\ll \alpha \ll \mu\ll \beta \le 1$ and let $T$ be an $n$-vertex tree with $\Delta(T)\le n^\alpha$. Suppose $G$ is a graph with at most $4n$ vertices that contains pairwise disjoint sets $U_1,U_2,V_1,V_2$ such that 
\begin{enumerate}[label=\upshape{(\roman{enumi})}]
    \item $|U_i|,|V_i|\ge (1-\mu)n/2$ for $i\in [2]$,
    \item $\delta(G[U_i,V_i])\ge (1-\mu)\min\{|U_i|,|V_i|\}$ for $i \in [2]$, and 
    \item $e(U_1)\ge 100\beta n^2$.
\end{enumerate}
If there is some $x_0\in V(G)$ such that we have $d(x_0,U_1\cup V_1)\ge \Delta(T)$ and $d(x_0,U_{2}\cup V_{2})\ge 1$, then $G$ contains a copy of $T$. 
\end{proposition}
\begin{proof}
Apply Lemma~\ref{lem:cutedge:0} with $\gamma=\frac{1}{100}$ to get an edge $zy\in E(T)$ and a subtree $T'\subset T$ with $y \in T'$ such that 
\begin{itemize}
    \item $n/100\Delta(T) \le |T'|\le n/100$, and
    \item $T'$ is a component of $T-e$. 
\end{itemize}
First, embed $z$ to $x_0$ and $y$ to the neighbour of $z$ in $U_2 \cup V_2$, and greedily complete the embedding of $T'$ into $U_2\cup V_2$. Then, use Lemma~\ref{lem:emb:extremal:0} to embed $T-T'$ into $U_1\cup V_1$ with $z$ copied to $x_0$. 
\end{proof}

\begin{proposition}\label{lem:emb:extremal:2}Let $1/n\ll \mu\ll  \Delta$ and let $T$ be an $n$-vertex tree with $\Delta(T)\le \Delta$. Suppose $G$ is a graph with at most $4n$ vertices that contains pairwise disjoint sets $U_1,U_2,V_1,V_2$ such that 
\begin{enumerate}[label=\upshape{(\roman{enumi})}]
    \item $|U_i|,|V_i|\ge (1-\mu)n/2$ for $i\in [2]$, and
    \item $\delta(G[U_i,V_i])\ge (1-\mu)\min\{|U_i|,|V_i|\}$.
\end{enumerate}
If there is some $x_0\in V(G)$ such that $d(x_0,U_1\cup V_1 \cup U_2)\ge 100\Delta(T)$ then $G$ contains a copy of $T$. 
\end{proposition}
\begin{proof}
    By \Cref{cor:t/2_forest} there is some vertex $z\in V(T)$ such that $T-z = F_1 \cup F_2$, where $F_1$ and $F_2$ are forests such that $F_1$ has at most $\lceil n/2 \rceil$ vertices and $F_2$ has a proper two-colouring such that the largest colour class has order at most $\frac{\Delta-1}{\Delta} \frac n2 + \frac 1 {\Delta}$. Embed $z$ to $x_0$ and embed $F_1$ and $F_2$ greedily to $U_1 \cup V_1$ and $U_2 \cup V_2$ respectively.
\end{proof}

\begin{proof}[Proof of Theorem~\ref{thm:main}]For $\Delta\ge 2$, we introduce  auxiliary constant $\beta,\mu>0$ so that $1/n\ll\mu\ll\beta\ll 1/\Delta$. Let $G$ be a blue/red coloured graph with $|G|=2n-1$ and $\delta(G)\ge \lfloor3|G|/4\rfloor$. Let $T$ be an $n$-vertex tree with $\Delta(T)\le \Delta$ and suppose $G$ contains no monochromatic copy of $T$. \stepcounter{propcounter}
\begin{claim}There is a partition $V(G)=V_0\cup V_1\cup V_2\cup V_3 \cup V_4$ such that the following properties hold. 
\begin{enumerate}[label=(\roman{enumi})]
    \item $|V_0|\le\mu n$ and $|V_i|=(1\pm\mu)n/2$ for $i\in [4]$.\label{it:V_i_size}
    \item For every $x\in V_{1}\cup V_{2}$ (resp. $x\in V_3\cup V_4$),  $\blue{d}(x,V_3\cup V_4)\le \mu n$ (resp. $\blue{d}(x,V_1\cup V_2)\le \mu n$).\label{it:low_blue_deg}
    \item For every $x\in V_{1}\cup V_{3}$ (resp. $x\in V_2\cup V_4$),  $\red{d}(x,V_2\cup V_4)\le \mu n$ (resp. $\red{d}(x,V_1\cup V_3)\le \mu n$).\label{it:low_red_deg} 
\end{enumerate}
    \end{claim}
    \begin{proof}[Proof of claim]We introduce an auxiliary constant $\mu'>0$ so that $1/n\ll\mu'\ll \mu$. First, we use Theorem~\ref{thm:stability} to get a partition $V(G)=U_0\cup U_1\cup U_2\cup U_3\cup U_4$ such that 
\begin{itemize}
  \item $|U_i|\ge (1-\mu')n/2$ for all $i\in [4]$,
    \item $\blue{e}(U_1\cup U_2,U_3\cup U_4)\le \mu' n^2$, and 
    \item $\red{e}(U_1\cup U_3, U_2\cup U_4)\le \mu' n^2$. 
\end{itemize}

Let $S_1 = \{x \in U_1 \cup U_2: \blue{d}(x, U_3 \cup U_4) \ge \mu n\}$. The pigeonhole principle implies that $|S_1| \le \frac{\mu'n^2}{\mu n}<2\mu'n$. Likewise, the following sets all have size at most $2\mu'n$:
\begin{align*}
    &S_2 = \{x \in U_3 \cup U_4: \blue{d}(x, U_1 \cup U_2) \ge \mu n\}, 
    \\&S_3 = \{x \in U_1 \cup U_3: \red{d}(x, U_2 \cup U_4) \ge \mu n\} \text{, and} 
    \\&S_4 =  \{x \in U_2 \cup U_4 : \red{d}(x, U_1 \cup U_3) \ge \mu n\}.
\end{align*} 

Let $V_0 = \bigcup_{i \in [4]}S_i \cup U_0$ and note $|V_0| \le 8 \mu'n + 2\mu'n  < \mu n$. For $i \in [4]$, let $V_i = U_i \setminus V_0$ and note $|V_i| = (1 \pm 21\mu')n/2 = (1 \pm \mu)n/2$, so \ref{it:V_i_size} is satisfied. The sets $V_i$ also satisfy \ref{it:low_blue_deg} and \ref{it:low_red_deg} by construction. 

\end{proof}
Note that each $x\in V_i$, $i\in [4]$, satisfies
\begin{equation}
    d(x,V_i)\ge \left\lfloor\frac{3|G|}{4}\right\rfloor-(1+\mu)n-2\mu n-|V_0|\ge (1-10\mu)\frac{n}{2}.
\end{equation}
Therefore, $e(V_1)\ge (1-100\mu)n^2/4$ and thus at $G[V_1]$ has least $(1-100\mu)n^2/8$ edges of the same colour, say red. \\

\noindent \textbf{Case 1.} $\blue{e}(V_i)\ge 100\beta n^2$ for some $i \in [4]$.\\ 

If some vertex $v\in V_0$ sends at least $100\Delta$ red edges to $V_1\cup V_3$ and at least 1 red edge to $V_2\cup V_4$, then \Cref{lem:emb:extremal:1} gives a red copy of $T$, so we assume any vertex of $V_0$ with a red neighbour in $V_2\cup V_4$ has less than $100 \Delta$ red neighbours in $V_1 \cup V_3$. Note that if such a vertex was contained in $V_2 \cup V_4$, then \ref{it:low_red_deg} would still hold. Likewise, if any vertex $v \in V_0$ has $100 \Delta$ blue neighbours in $V_1 \cup V_2$ (resp. $V_3 \cup V_4$), then we may assume it has at most $100 \Delta$ blue neighbours in $V_3 \cup V_4$ (resp. $V_1 \cup V_2$) and if it was contained in $V_1 \cup V_2$ (resp. $V_3 \cup V_4$), then \ref{it:low_blue_deg} would still hold. 

This reasoning allows us to assume that either $G$ contains a monochromatic copy of $T$, or there exists a partition $V(G) = V_1'\cup  V_2'\cup V_3'\cup V_4'$ satisfying \ref{it:low_blue_deg} and \ref{it:low_red_deg}, and $|V_i'| = (1 \pm 3\mu)n/2$ for $i \in [4]$. To be precise, each $v \in V(G)$ satisfies one of the following:
\begin{itemize}
    \item $v$ has a red neighbour in $V_2 \cup V_4$ and at least $100 \Delta $ blue neighbours in $V_1 \cup V_2$. Then $v \in V_2'$, or
    \item $v$ has a red neighbour in $V_2 \cup V_4$ and less than $100 \Delta$ blue neighbours in $V_1 \cup V_2$. Then $v \in V_4'$, or
    \item $v$ has no red neighbours in $V_2 \cup V_4$ and at least $100 \Delta$ blue neighbours in $V_1 \cup V_2$. Then $v \in V_1'$, or
    \item $v$ has no red neighbours in $V_2 \cup V_4$ and less than $100 \Delta$ blue neighbours in $V_1 \cup V_2$. Then $v \in V_3'$.
\end{itemize}
As $V_1 \se V_1'$, we have that $G[V_1']$ has at least $(1-100\mu)n^2/8$ red edges, and  $\blue{e}(V_i')\ge 100\beta n^2$ for some $i \in [4]$, if there is any red edge from $V_1' \cup V_3'$ to $V_2' \cup V_4'$ or a blue edge from $V_1' \cup V_2'$ to $V_3' \cup V_4'$, then we can conclude by \Cref{lem:emb:extremal:1}.

By the order of $G$, there is some $j \in [4]$ such that $|V_j'| \ge \lfloor n/2 \rfloor$. We assume $j = 4$. Then 
\[|G \setminus V_4'| \le 2n-1 - \lfloor n/2 \rfloor = n + \lceil n/2 \rceil -1.\]
Thus, let $v \in V_1'$ and note $v$ does not have an edge to itself, by the minimum degree condition we have $d(v, V_4) \ge \lfloor 3|G|/4 \rfloor - (n + \lceil n/2 \rceil -2) > 0$. Thus there is some edge between $V_1'$ and $V_4'$, which either acts as a red edge between $V_1' \cup V_3'$ and $V_2'\cup V_4'$ or a blue edge between $V_1' \cup V_2'$ and $V_3' \cup V_4'$, and we conclude by \Cref{lem:emb:extremal:1}.\\

% If some vertex $v\in V_0$ has at least $100\Delta(T)$ red neighbours in both $V_1\cup V_3$ and $V_2\cup V_4$, then Lemma~\ref{lem:emb:extremal:1} gives a red copy of $T$. Thus, we may partition $V_0=V_{1,3}\cup V_{2,4}$ so that every $x\in V_{1,3}$ satisfies $\red{d}(x,V_2\cup V_4\cup V_{2,4})\le 3\mu n$, and every $x\in V_{2,4}$ satisfies $\red{d}(x,V_1\cup V_3\cup V_{1,3})\le 3\mu n$. 

% \textcolor{red}{Note that as $|G|=2n-1$, either $|V_1\cup V_3\cup V_{13}|\ge n$ or $|V_2\cup V_4\cup V_{14}|\ge n$, so let us assume without of generality that $|V_1\cup V_3\cup V_{13}|\ge n$. Cover $V_{13}$ using Lemma~\ref{lem:covering:2} and use Lemma~\ref{lem:emb:extremal:2} to complete the embedding of $T$.} \\

\noindent \textbf{Case 2.} $\blue{e}(V_i)< 100\beta n^2$ for all $i \in [4]$.    \\

If some $v \in V_0$ has at least $100 \Delta$ blue neighbours in at least three of $V_1, V_2, V_3$ and $V_4$, then we can embed $T$ using \Cref{lem:emb:extremal:2}. Otherwise, each vertex $v \in V_0$ has 
\begin{equation}\label{eq:vertex_degs}
    \red{d}(x, V_i \cup V_j) \ge \lfloor 3/4|G|\rfloor- (1+\mu)n - 100 \Delta > (1 - 3 \mu)n/2,
\end{equation}
for $(i, j) \in \{(1,3), (2,4)\}$. We partition $V_0 = V_{1,3} \cup V_{2,4}$ by assigning each vertex of $V_0$ to a set $V_{i, j}$ for which \eqref{eq:vertex_degs} holds. We now have a partition $V(G) = V_1 \cup V_3 \cup V_{1,3} \cup V_2 \cup V_4 \cup V_{2,4}$. We let $V'$ be the larger of $V_1 \cup V_3 \cup V_{1,3} $ and $ V_2 \cup V_4 \cup V_{2,4}$. Then $|V'| \ge n$ and $\delta_{red}(V') \ge (1-3\mu)n/2$, and all but at most $\mu n$ vertices $v\in V'$ have 
\[\red{d}(v, V')\ge \lfloor 3n/2 \rfloor -1 - (1 + \mu )n/2-2\mu n \ge (1-6\mu)n.\] Thus we may conclude by Lemma \ref{lem:exact:1}.

%If some vertex $v\in V_0$ has at least $100\Delta(T)$ red neighbours in both $V_1\cup V_3$ and $V_2\cup V_4$, then Lemma~\ref{lem:emb:extremal:1} gives a red copy of $T$. Thus, we may partition $V_0=V_{1,3}\cup V_{2,4}$ so that every $x\in V_{1,3}$ satisfies $\red{d}(x,V_2\cup V_4\cup V_{2,4})\le 3\mu n$, and every $x\in V_{2,4}$ satisfies $\red{d}(x,V_1\cup V_3\cup V_{1,3})\le 3\mu n$. 
\end{proof}

\begin{proof}[Proof of Theorem~\ref{thm:main3}]Let $\Delta\ge 2$ and let $C$ be sufficiently large compared to $\Delta$. We introduce auxiliary constants $\beta,\mu>0$ so that $1/n\ll\mu\ll \beta\ll 1/\Delta$, and let $G$ be a blue/red coloured graph with $|G|=2n-2$ and $\delta(G)\ge \lfloor3|G|/4\rfloor+C$. Let $T$ be an $n$-vertex tree with $\Delta(T)\le \Delta$ and suppose $G$ contains no monochromatic copy of $T$. Similarly as for the proof of Theorem~\ref{thm:main}, we may assume that there is a partition $V(G)=V_0\cup V_1\cup V_2\cup V_3\cup V_4$ such that 
\begin{enumerate}[label=(\roman{enumi})]
    \item $|V_0|\le\mu n$ and $|V_i|=(1\pm\mu)n/2$ for $i\in [4]$.\label{it:V_i_size2}
    \item For every $x\in V_{1}\cup V_{2}$ (resp. $x\in V_3\cup V_4$),  $\blue{d}(x,V_3\cup V_4)\le \mu n$ (resp. $\blue{d}(x,V_1\cup V_2)\le \mu n$).\label{it:low_blue_deg2}
    \item For every $x\in V_{1}\cup V_{3}$ (resp. $x\in V_2\cup V_4$),  $\red{d}(x,V_2\cup V_4)\le \mu n$ (resp. $\red{d}(x,V_1\cup V_3)\le \mu n$).\label{it:low_red_deg2} 
\end{enumerate}
Again, for each $x\in V_i$, $i\in [4]$, we have
\begin{equation*}
    d(x,V_i)\ge \left\lfloor\frac{3|G|}{4}\right\rfloor+C-(1+\mu)n-2\mu n-|V_0|\ge (1-10\mu)\frac{n}{2}.
\end{equation*}
In particular, we have $e(V_1)\ge (1-100\mu)n^2/4$ and thus at $G[V_1]$ has least $(1-100\mu)n^2/8$ edges of the same colour, say red. Similar to the proof of Theorem~\ref{thm:main}, we have two cases: either $\blue{e}(V_i)\ge 100\beta n^2 $ for some $i\in [4]$, or $\blue{e}(V_i)\le 100\beta n^2$ for all $i\in [4]$. The former case follows by the same reasoning as in the proof of Theorem~\ref{thm:main}, so let us assume that the second case holds, that is, we have $\blue{e}(V_i)\le 100\beta n^2$ for all $i\in [4]$. 

We note that if some $v \in V_0$ has at least 1 red edge to $V_1 \cup V_3$, then we may assume that it has at most $100 \Delta$ red edges to $V_2 \cup V_4$, otherwise we may conclude by Proposition~\ref{lem:emb:extremal:1}. Thus we partition $V_0 = V_{1, 3} \cup V_{2, 4}$ so that each vertex of $V_{1, 3}$ (resp $V_{2, 4}$) satisfies $\red{d}(x, V_2 \cup V_4) \le 100 \Delta$ (resp $\red{d}(x, V_1 \cup V_3) \le 100 \Delta$). 

Now observe that if some vertex $v \in V_0$ sends at least $100 \Delta$ blue edges to three of $V_1, V_2, V_3, V_4$, we may conclude by Proposition~\ref{lem:emb:extremal:2}. Hence each $x \in V_{1, 3}$ satisfies
\[
\red{d}(x, V_1 \cup V_3)  
\ge 
\left\lfloor\frac{3(2n-2)}{4}\right\rfloor - (1 + 3\mu)n 
\ge 
(1 - 7 \mu )\frac{n}{2}. 
\]
Thus if $e(V_{1,3}, V_2 \cup V_4) \ge 1$, we can conclude by Proposition~\ref{lem:emb:extremal:1}. A symmetric argument also implies $e(V_{2,4}, V_1 \cup V_3) = 0$. Next, note each vertex $x \in V_{1,3}$ has $\blue{d}(x, V_{j}) \ge (1 - 5\mu)n/4$ for some $j \in \{2, 4\}$, as otherwise we have
\[d(x) \le |V_1 \cup V_3|  + |V_{0}|  + \blue{d}(x, V_2) + \blue{d}(x, V_4) \le (1 + 2\mu)n  + (1 - 5\mu)n/2< \delta(G), \]
a contradiction. A symmetric argument holds for $V_{2,4}$ We now partition $V_{1,3} = V_1' \cup V_3'$ and $V_{2, 4} = V_2' \cup V_4'$ by assigning each $x \in V_{1, 3}$ to $V_i'$, where $\blue{d}(x, V_{i + 1})$ is maximal over $ i \in [2]$, and each $x \in V_{2,4}$ to $V_i'$, where $\blue{d}(x, V_{i -1})$ is maximal over $ i \in [2]$. 

We let $W_i = V_i \cup V_i'$ for all $i \in [4]$ and note that $\red{e}(W_1 \cup W_3, W_2 \cup W_4) = 0$, otherwise we may conclude by Proposition~\ref{lem:emb:extremal:1}. If $|W_i \cup W_j| \ge n$ for some $(i, j) \in \{(1, 3), (2, 4)\}$ then we conclude Lemma~\ref{lem:exact:1}. Otherwise,  we must have that $|W_1\cup W_3|=n-1$ and $|W_2\cup W_4|=n-1$. Note that for $x\in W_1$, we have 
\[
\blue{d}(x,W_4)
\ge
\left\lfloor\frac{3(2n-2)}{4}\right\rfloor+C-(|W_1\cup W_3|-1)-|W_2|=
\left\lfloor\frac{n+1}{2}\right\rfloor-|W_2|+C.
\]
Similarly, for $x\in W_3$ we have that $d(x,W_2)\ge \lfloor\frac{n+1}{2}\rfloor-|W_4|+C$. As $|W_2\cup W_4|=n-1$, then either $|W_2|\le \lfloor\frac{n+1}{2}\rfloor$ or $|W_4|\le \lfloor\frac{n+1}{2}\rfloor$. Suppose without loss of generality that $|W_2|\le \lfloor\frac{n+1}{2}\rfloor$. In this case, we have 
\[\blue{e}(W_1,W_4)\ge C|W_1|\ge \frac{|C|}{4}|W_1\cup W_4|,\]
in which case we can conclude by Lemma~\ref{lem:emb:extremal:3}.

% If $|V_i \cup V_j| \ge n$ for some $(i, j) \in \{(1, 3), (2, 4)\}$ then we conclude Lemma~\ref{lem:exact:1}. Otherwise,  we must have that $|V_1\cup V_3|=n-1$ and $|V_2\cup V_4|=n-1$. Note that for $x\in V_1$, we have 
% \[\blue{d}(x,V_4)\ge \lfloor\frac{3(2n-2)}{4}\rfloor+C-(|V_1\cup V_3|-1)-|V_3|=\lfloor\frac{n+1}{2}\rfloor-|V_3|+C.\]
% Similarly, for $x\in V_3$ we have that $d(x,V_2)\ge \lfloor\frac{n+1}{2}\rfloor-|V_1|+C$. As $|V_1\cup V_3|=n-1$, then either $|V_1|\le \lfloor\frac{n+1}{2}\rfloor$ or $|V_3|\le \lfloor\frac{n+1}{2}\rfloor$. Suppose without loss of generality that $|V_3|\le \lfloor\frac{n+1}{2}\rfloor$. In this case, we have 
% \[\blue{e}(V_1,V_4)\ge C|V_1|\ge \frac{|C|}{4}|V_1\cup V_2|,\]
% in which case we can conclude by Lemma~\ref{lem:emb:extremal:3}.

%\[\delta(\blue{G}[V_1, V_4]) \ge \delta(G) - 3(n-1)/2 - 400 \Delta \ge C - 400 \Delta.\]
%Likewise, $\delta(\blue{G}[V_2, V_3]) \ge C-400 \Delta$ and we conclude by Proposition~\ref{lem:emb:extremal:3}.

\end{proof}

\begin{proof}[Proof of Theorem~\ref{thm:main2}]
    We introduce an auxiliary constant $\mu$ that satisfies $0<\alpha \ll\mu\ll\delta$. Let $G$ be a blue/red coloured graph with $2n-2$ vertices and minimum degree $\delta(G)\ge (3/4+\delta)2n$, and let $T$ be an $n$-vertex tree with $\Delta(T)\le n^{\alpha}$. If $G$ contains no monochromatic copy of $T$, then, by Theorem~\ref{thm:stability:unbounded}, we may assume that there is a partition $V(G)=U_0\cup U_1\cup U_2$ such that 
\begin{itemize}
    \item $|U_i|\ge (1-\mu)n$ for $i\in [2]$,
    \item $\delta(\red{G}[U_i])\ge n/2+\delta n/100$ for $i\in [2]$, and 
    \item $\red{e}(U_1,U_2)\le \mu n^2$.
\end{itemize}
\begin{claim}\label{thm2:claim1}If there is a vertex $x_0$ such that $\red{d}(x_0, U_i)\ge 100\mu n$ for $i\in [2]$, then $T\subset \red{G}$.
\end{claim}
\begin{proof}[Proof of claim]Let $\gamma>0$ be an auxiliary constant such that $\mu\ll\gamma\ll \delta$. Using Lemma~\ref{lem:cutvertex:2}, we find a vertex $z\in V(T)$ and subtrees $T_1,T_2\subset T$ such that $\gamma n\le |T_1|\le 2\gamma n$, $V(T_1)\cap V(T_2)=\{z\}$ and $T=T_1\cup T_2$. Noting that $|T_2|\le n-\gamma n+1\le (1-\gamma/2)n\le |U_1|$, as $\mu\ll \gamma$, and that $\delta(\red{G}[U_1])\ge (1+\delta/100)|U_1|/2$, we can use Theorem~\ref{thm:KSS} to find a copy of $T_2$ in $\red{G}[U_1]$ with $z$ copied to $x_0$. Then, complete a copy of $T$ by greedily mapping $T_1$ to $\red{G}[U_2]$, starting at $x_0$, using that $\delta(\red{G}[U_2])\ge n/2\ge |T_1|$ and $\Delta(T)\le n^{\alpha}\le 100\mu n$. 
\end{proof}
%\begin{claim}\label{thm2:claim2}If there is a vertex $x_0\in U_0$ such that $\blue{d}(x_0, U_i)\ge 100\mu n$ for $i\in [2]$, then $T\subset \blue{G}$.    
%\end{claim}
%\begin{proof}[Proof of claim]
    %\end{proof}
Note that Claim~\ref{thm2:claim1} implies each vertex $u\in U_0$ sends at most $100\mu n$ red edges to either $U_1$ or $U_2$. Therefore, we may partition $U_0=U_{0,1}\cup U_{1,1}$ so that, for $i\in [2]$, 
\begin{itemize}
    \item each vertex $u\in U_{0,i}$ satisfies $\red{d}(u,U_{2-i})< 100\mu n$.
\end{itemize}
Defining $U_1^+=U_1\cup U_{0,1}$ and $U_2^+=U_2\cup U_{0,2}$, we have that each vertex $x\in U^+_{1}$ satisfies 
\begin{eqnarray*}
    \blue{d}(x,U^+_2)&\ge& (3/4+\delta)2n-|U_1|-|U_0|-100\mu n\\&\ge& (3/4+\delta)2n-(2n-1-|U_2|)-100\mu n \\
    &\ge& (1+\delta/2)|U_2^+|/2,
\end{eqnarray*}
and similarly for a vertex in $U_2^+$. Therefore, as we may assume $|U_1^+|\ge |U_2^+|$, we have $|U_1^+|\ge n-1$ and $|U_2^+|\ge (1-\mu)n$ and thus we can use Lemma~\ref{thm:KSS:bipartite} to find a copy of $T$ in $\blue{G}[U_1^+,U_2^+]$.

\end{proof}

\section{Concluding remarks}\label{sec:remarks}
While Theorem \ref{thm:main} resolves the conjecture of Schelp~\cite{schelp2012}, one could ask for additional improvements.

\paragraph*{Improving the minimum degree in Theorem~\ref{thm:main3}}
From the proof of Theorem~\ref{thm:main3}, we see that we can take the constant $C$ of the form $C=\text{poly}(\Delta)$. Ideally, one would show that $C=0$ suffices, though even proving $C=\Delta^{1+o(1)}$ would be interesting.

\paragraph*{Larger maximum degree}
As mentioned in the introduction, we cannot expect to improve the maximum degree condition in Theorems~\ref{thm:main} and~\ref{thm:main2} beyond $o(n/\log n)$. In this direction, we believe that Theorem~\ref{thm:main2} should hold for trees of higher maximum degree.
\begin{question}Prove that Theorem~\ref{thm:main2} holds for trees with maximum degree $cn/\log n$, provided $c$ is sufficiently small. \end{question}
Although our approach does not work for this question, because we rely heavily on controlling the maximum degree to prove Theorems~\ref{thm:stability} and~\ref{thm:stability:unbounded}, we believe that the techniques we used for the stability analysis might still be useful. 

\paragraph*{Smaller host graph}
 Burr~\cite{burr74} noted that the Ramsey number of a tree $T$ seems to depend not only on the number of vertices but rather on the sizes of the bipartition classes of $T$, and conjectured that if $T$ has bipartition classes of sizes $t_1,t_2$ with $t_1\ge t_2\ge 2$, then $R(T)=\max\{t_1+2t_2,2t_1\}-1$. Though Burr's conjecture was known to be false for certain double stars (see e.g.~\cite{dubo2025ramsey,norinsunzhao}), Montgomery, Pavez-Sign\'e and Yan~\cite{montgomery2025ramsey} recently showed that Burr's conjecture is indeed true for $n$-vertex trees with maximum degree at most $cn$, where $c\in (0,1)$ is some small absolute constant. 
As our extremal example has size $2n-2$ and $2n-2\ge \max\{t_1 + 2t_2, 2t_1\}-1 $ holds for all $t_1\ge t_2\ge 2$ with $t_1+t_2=n$, it might be possible to have a version of Theorem~\ref{thm:main} for smaller host graphs.
\begin{question}Let $\Delta \in \NN$ and let $T$ be a tree with $\Delta(T)\le \Delta$ and parts of sizes $t_1,t_2$ such that $t_1\ge t_2\ge 2$. Is it true that if  $G$ is a graph with $|G|=\max\{t_1 + 2t_2, 2t_1\}-1$ and minimum degree $\delta(G)\ge\lfloor3|G|/4\rfloor+C_\Delta$, then every blue/red colouring of $G$ contains a monochromatic copy of $T$?
\end{question}

\medskip

\bibliographystyle{habbrv}
\bibliography{Citations}

@article{campos2023exponential,
  title={An exponential improvement for diagonal {R}amsey},
  author={Campos, M. and Griffiths, S. and Morris, R. and Sahasrabudhe, J.},
   JOURNAL = {Ann. of Math. (2)},
  FJOURNAL = {Annals of Mathematics. Second Series},
  year={2025},
doi={10.48550/arXiv.2303.09521},
note={to appear}
}

@article{BONDY197346,
author = {J.A. Bondy and P. Erd\H{o}s},
     TITLE = {Ramsey numbers for cycles in graphs},
   JOURNAL = {J. Combin. Theory Ser. B},
  FJOURNAL = {Journal of Combinatorial Theory. Series B},
    VOLUME = {14},
      YEAR = {1973},
     PAGES = {46--54},
       DOI = {10.1016/s0095-8956(73)80005-x},
}

@article{faudree1974all,
  title={All {R}amsey numbers for cycles in graphs},
  author={Faudree, R. J. and Schelp, R. H.},
   JOURNAL = {Discrete Math.},
  FJOURNAL = {Discrete Mathematics},
  volume={8},
  number={4},
  pages={313--329},
  year={1974},
       DOI = {10.1016/0012-365X(74)90151-4},
}

@article{gerencser1967ramsey,
    AUTHOR = {Gerencs\'er, L. and Gy\'arf\'as, A.},
     TITLE = {On {R}amsey-type problems},
   JOURNAL = {Ann. Univ. Sci. Budapest. E\"otv\"os Sect. Math.},
  FJOURNAL = {Annales Universitatis Scientiarum Budapestinensis de Rolando
              E\"otv\"os Nominatae. Sectio Mathematica},
    VOLUME = {10},
      YEAR = {1967},
     PAGES = {167--170},
      ISSN = {0524-9007},
}

@incollection{McDiarmid_1989, 
    author={McDiarmid, C.},
     TITLE = {On the method of bounded differences},
 BOOKTITLE = {Surveys in combinatorics, 1989 ({N}orwich, 1989)},
    SERIES = {London Math. Soc. Lecture Note Ser.},
    VOLUME = {141},
     PAGES = {148--188},
 PUBLISHER = {Cambridge Univ. Press, Cambridge},
      YEAR = {1989},
}

@article{NIKIFOROV200869,
title = {Cycles and stability},
author = {V. Nikiforov and R.H. Schelp},
   JOURNAL = {J. Combin. Theory Ser. B},
  FJOURNAL = {Journal of Combinatorial Theory. Series B},
volume = {98},
number = {1},
pages = {69-84},
year = {2008},
doi = {10.1016/j.jctb.2007.05.001},
}

@article{balogh2025ramsey,
  title={Ramsey-type problems for tilings in dense graphs},
  author={Balogh, J. and Freschi, A. and Treglown, A.},
   JOURNAL = {European J. Combin.},
  FJOURNAL = {European Journal of Combinatorics},
volume = {131},
pages = {Paper No. 104228},
year = {2026},
doi = {10.1016/j.ejc.2025.104228},
}

@article{aragao2025degree,
  author={Arag{\~a}o, L. and Marciano, J. P. and Mendon{\c{c}}a, W.},
    TITLE = {Degree conditions for {R}amsey goodness of paths},
   JOURNAL = {European J. Combin.},
  FJOURNAL = {European Journal of Combinatorics},
    VOLUME = {124},
      YEAR = {2025},
     PAGES = {Paper No. 104082},
       DOI = {10.1016/j.ejc.2024.104082},
}

@article{balogh2022monochromatic,
  title={Monochromatic paths and cycles in 2-edge-coloured graphs with large minimum degree},
  author={Balogh, J. and Kostochka, A. and Lavrov, M. and Liu, X.},
 JOURNAL = {Combin. Probab. Comput.},
  FJOURNAL = {Combinatorics, Probability and Computing},
  volume={31},
  number={1},
  pages={109--122},
  year={2022},
       DOI = {10.1017/s0963548321000201},
}

@Incollection{Harary1972,
author="Harary, F.",
     TITLE = {Recent results on generalized {R}amsey theory for graphs},
 BOOKTITLE = {Graph theory and applications ({P}roc. {C}onf., {W}estern
              {M}ichigan {U}niv., {K}alamazoo, {M}ich., 1972; dedicated to
              the memory of {J}. {W}. {T}. {Y}oungs)},
    SERIES = {Lecture Notes in Math.},
    VOLUME = {Vol. 303},
     PAGES = {125--138},
 PUBLISHER = {Springer, Berlin-New York},
      YEAR = {1972},
}

@book{JLR2000,
	    AUTHOR = {Janson, S. and {\L}uczak, T. and Ruci{\'n}ski, A.},
	TITLE = {Random graphs},
    SERIES = {Wiley-Interscience Series in Discrete Mathematics and
              Optimization},
 PUBLISHER = {Wiley-Interscience, New York},
	YEAR = {2000},
	PAGES = {xii+333},
	ISBN = {0-471-17541-2},
	MRCLASS = {05C80 (60C05 82B41)},
	MRNUMBER = {1782847},
	MRREVIEWER = {Mark R. Jerrum},
	DOI = {10.1002/9781118032718},
}

@article{rosta1973ramsey,
     TITLE = {On a {R}amsey-type problem of {J}. {A}. {B}ondy and {P}.
              {E}rd{\H{o}}s. {I}, {II}},
  author={Rosta, Vera},
   JOURNAL = {J. Combinatorial Theory Ser. B},
  FJOURNAL = {Journal of Combinatorial Theory. Series B},
  volume={15},
  number={1},
     PAGES = {94--104; ibid. 15 (1973), 105--120},
  year={1973},
       DOI = {10.1016/0095-8956(73)90035-x},
}

@article{bucic2023tight,
  author={Buci{\'{c}}, M. and Sudakov, B.},
     TITLE = {Tight {R}amsey bounds for multiple copies of a graph},
   JOURNAL = {Adv. Comb.},
  FJOURNAL = {Advances in Combinatorics},
      YEAR = {2023},
month = {3 Feb},
volume={1},
     PAGES = { 22pp.},
	doi = {10.19086/aic.2023.1},
}

@article{Besomi2019,
    AUTHOR = {Besomi, G. and Pavez-Sign\'e, M. and Stein, M.},
     TITLE = {Degree conditions for embedding trees},
   JOURNAL = {SIAM J.~Discrete Math.},
  FJOURNAL = {SIAM Journal on Discrete Mathematics},
    VOLUME = {33},
      YEAR = {2019},
    NUMBER = {3},
     PAGES = {1521--1555},
      ISSN = {0895-4801,1095-7146},
   MRCLASS = {05C35},
  MRNUMBER = {3997713},
       DOI = {10.1137/18M1210861},
}

@article{Bondy1971,
    AUTHOR = {Bondy, J. A.},
     TITLE = {Pancyclic graphs. {I}},
   JOURNAL = {J. Combin. Theory Ser. B},
  FJOURNAL = {Journal of Combinatorial Theory. Series B},
    VOLUME = {11},
      YEAR = {1971},
     PAGES = {80--84},
      ISSN = {0095-8956},
       DOI = {10.1016/0095-8956(71)90016-5},
}

@article{Dirac52,
   AUTHOR = {Dirac, G. A.},
     TITLE = {Some theorems on abstract graphs},
   JOURNAL = {Proc. London Math. Soc. (3)},
  FJOURNAL = {Proceedings of the London Mathematical Society. Third Series},
    VOLUME = {2},
      YEAR = {1952},
     PAGES = {69--81},
      ISSN = {0024-6115,1460-244X},
       DOI = {10.1112/plms/s3-2.1.69},
}

@article{zhao2011,
    AUTHOR = {Zhao, Y.},
     TITLE = {Proof of the {$(n/2-n/2-n/2)$} conjecture for large {$n$}},
   JOURNAL = {Electron. J. Combin.},
  FJOURNAL = {Electronic Journal of Combinatorics},
    VOLUME = {18},
      YEAR = {2011},
    NUMBER = {1},
     PAGES = {Paper 27, 61pp.},
       DOI = {10.37236/514},
}

@article{schelp2012,
    AUTHOR = {Schelp, R. H.},
     TITLE = {Some {R}amsey-{T}ur\'an type problems and related questions},
   JOURNAL = {Discrete Math.},
  FJOURNAL = {Discrete Mathematics},
    VOLUME = {312},
      YEAR = {2012},
    NUMBER = {14},
     PAGES = {2158--2161},
       DOI = {10.1016/j.disc.2011.09.015},
}

@article{li2010new,
    AUTHOR = {Li, H. and Nikiforov, V. and Schelp, R. H.},
     TITLE = {A new class of {R}amsey-{T}ur\'an problems},
   JOURNAL = {Discrete Math.},
  FJOURNAL = {Discrete Mathematics},
    VOLUME = {310},
      YEAR = {2010},
    NUMBER = {24},
     PAGES = {3579--3583},
       DOI = {10.1016/j.disc.2010.09.009},
}

@article{cycles_2012, 
    AUTHOR = {Benevides, F. S. and {\L}uczak, T. and Scott, A. and Skokan, J.
              and White, M.},
     TITLE = {Monochromatic cycles in 2-coloured graphs},
   JOURNAL = {Combin. Probab. Comput.},
  FJOURNAL = {Combinatorics, Probability and Computing},
    VOLUME = {21},
      YEAR = {2012},
    NUMBER = {1-2},
     PAGES = {57--87},
    DOI={10.1017/S0963548312000090}, 
}

@incollection{wormald1999differential,
  title={The differential equation method for random graph processes and greedy algorithms},
  author={Wormald, Nicholas C},
editor = {M. Karoński and H. J. Prömel},
  booktitle={Lectures on approximation and randomized algorithms},
publisher = "Wydawnictwo Naukowe Pwn",
  pages={73--155},
  year={1999}
}

@article{luczak1999r,
  title={${R}({C}_n, {C}_n, {C}_n)\le(4+ o (1)) n$},
  author={{\L}uczak, Tomasz},
   JOURNAL = {J. Combin. Theory Ser. B},
  FJOURNAL = {Journal of Combinatorial Theory. Series B},
  volume={75},
  number={2},
  pages={174--187},
  year={1999},
       DOI = {10.1006/jctb.1998.1874},
}

@article{kathapurkar2022spanning,
  title={Spanning trees in dense directed graphs},
  author={Kathapurkar, Amarja and Montgomery, Richard},
   JOURNAL = {J. Combin. Theory Ser.~B},
  FJOURNAL = {Journal of Combinatorial Theory. Series B},
  volume={156},  
  pages={223--249},
  year={2022},
       DOI = {10.1016/j.jctb.2022.04.007},
}

@article{bottcher2020embedding,
  title={Embedding spanning bounded degree graphs in randomly perturbed graphs},
  author={B{\"o}ttcher, Julia and Montgomery, Richard and Parczyk, Olaf and Person, Yury},
  journal={Mathematika},
  volume={66},
  number={2},
  pages={422--447},
  year={2020},
       DOI = {10.1112/mtk.12005},
}

@article{bottcher2012tripartite,
  title={The tripartite {R}amsey number for trees},
  author={B{\"o}ttcher, Julia and Hladk{\`y}, Jan and Piguet, Diana},
   JOURNAL = {J. Graph Theory},
  FJOURNAL = {Journal of Graph Theory},
  volume={69},
  number={3},
  pages={264--300},
  year={2012},
       DOI = {10.1002/jgt.20582},
}

@article{dubo2025ramsey,
  title={On the {R}amsey number of the double star},
  author={Dub{\'o}, Freddy Flores and Stein, Maya},
   JOURNAL = {Discrete Math.},
  FJOURNAL = {Discrete Mathematics},
    VOLUME = {348},
      YEAR = {2025},
    NUMBER = {1},
     PAGES = {Paper No.~114227, 4},
       DOI = {10.1016/j.disc.2024.114227},
}

@article{paths_schelp_conj,
    AUTHOR = {Gy\'arf\'as, A. and S\'ark\"ozy, G. N.},
     TITLE = {Star versus two stripes {R}amsey numbers and a conjecture of
              {S}chelp},
   JOURNAL = {Combin. Probab. Comput.},
  FJOURNAL = {Combinatorics, Probability and Computing},
    VOLUME = {21},
      YEAR = {2012},
    NUMBER = {1-2},
     PAGES = {179--186},
       DOI = {10.1017/S0963548311000599},
       URL = {https://doi.org/10.1017/S0963548311000599},
}

@article{KSS, 
     AUTHOR = {Koml\'os, J. and S\'ark\"ozy, G. N. and Szemer\'edi, E.},
     TITLE = {Spanning trees in dense graphs},
   JOURNAL = {Combin. Probab. Comput.},
  FJOURNAL = {Combinatorics, Probability and Computing},
    VOLUME = {10},
      YEAR = {2001},
    NUMBER = {5},
     PAGES = {397--416},
       DOI = {10.1017/S0963548301004849},
       URL = {https://doi.org/10.1017/S0963548301004849}
}

@inproceedings{burr74,
    AUTHOR = {Burr, S. A.},
     TITLE = {Generalized {R}amsey theory for graphs---a survey},
 BOOKTITLE = {Graphs and combinatorics ({P}roc. {C}apital {C}onf., {G}eorge
              {W}ashington {U}niv., {W}ashington, {D}.{C}., 1973)},
    SERIES = {Lecture Notes in Math.},
    VOLUME = {Vol. 406},
     PAGES = {52--75},
 PUBLISHER = {Springer, Berlin-New York},
      YEAR = {1974},
}

@article{BurrErdosRamsey,
    AUTHOR = {Burr, S. A. and Erd\H{o}s, P.},
     TITLE = {Extremal {R}amsey theory for graphs},
   JOURNAL = {Utilitas Math.},
  FJOURNAL = {Utilitas Mathematica. An International Journal of Discrete and
              Combinatorial Mathematics, and Statistical Design},
    VOLUME = {9},
      YEAR = {1976},
     PAGES = {247--258},
}

@article {ramsey30,
    AUTHOR = {Ramsey, F. P.},
     TITLE = {On a {P}roblem of {F}ormal {L}ogic},
   JOURNAL = {Proc. London Math. Soc. (2)},
  FJOURNAL = {Proceedings of the London Mathematical Society. Second Series},
    VOLUME = {30},
      YEAR = {1929},
    NUMBER = {4},
     PAGES = {264--286},
      ISSN = {0024-6115},
   MRCLASS = {99-04},
  MRNUMBER = {1576401},
       DOI = {10.1112/plms/s2-30.1.264},
       URL = {https://doi.org/10.1112/plms/s2-30.1.264},
}

@misc{pokrovskiy2025embedding,
  title={Embedding trees using minimum and maximum degree conditions},
  author={Pokrovskiy, Alexey and Versteegen, Leo and Williams, Ella},
  year={2025},
      eprint={arXiv:2512.16799},
      archivePrefix={arXiv},
      primaryClass={math.CO},
   doi={10.48550/arXiv.2512.16799}, 
}

@misc{chen2025embedding,
  title={Embedding loose trees in {$k$}-uniform hypergraphs},
  author={Chen, Y. and Lo, A.},
  year={2025},
      eprint={arXiv:2502.04783},
      archivePrefix={arXiv},
      primaryClass={math.CO},
   doi={10.48550/arXiv.2502.04783}, 
}

@misc{norinsunzhao,
      title={Asymptotics of {R}amsey numbers of double stars}, 
      author={S. Norin and Y. R. Sun and Y. Zhao},
      year={2016},
      eprint={arxiv:1605.03612},
      archivePrefix={arXiv},
      primaryClass={math.CO},
   doi={10.48550/arXiv.1605.03612}, 
}

@misc{balister2024upper,
  title={Upper bounds for multicolour {R}amsey numbers},
  author={Balister, P. and Bollob{\'a}s, B. and Campos, M. and Griffiths, S. and Hurley, E. and Morris, R. and Sahasrabudhe, J. and Tiba, M.},
  year={2024},
      eprint={arxiv:2410.17197},
      archivePrefix={arXiv},
      primaryClass={math.CO},
doi={10.48550/arXiv.2410.17197},
}

@misc{montgomery2025ramsey,
  title={Ramsey numbers of trees},
  author={Montgomery, R. and Pavez-Sign{\'e}, M. and Yan, J.},
  year={2025},
      eprint={arxiv:2509.07934},
      archivePrefix={arXiv},
      primaryClass={math.CO},
doi={10.48550/arXiv.2509.07934},
}

@misc{gupta2024optimizing,
  title={Optimizing the {CGMS} upper bound on {R}amsey numbers},
  author={Gupta, P. and Ndiaye, N. and Norin, S. and Wei, L.},
  year={2024},
      eprint={arxiv:2407.19026},
      archivePrefix={arXiv},
      primaryClass={math.CO},
doi={10.48550/arXiv.2407.19026},
}

\end{document}